\documentclass{amsart} 
\usepackage{amssymb,times, amscd, amsmath, amsthm, xypic}
\usepackage[dvips]{graphicx, color}
\author{J\"{o}rg Sch\"{u}rmann and Shoji Yokura}
\address
{J\"{o}rg Sch\"{u}rmann:
Westf. Wilhelms-Universit\"{a}t,
Mathematisches Institut, Einsteinstrasse 62,
48149 M\"{u}nster, Germany }
\email {jschuerm@math.uni-muenster.de}

\address
{Department of Mathematics and Computer Science, 
Faculty of Science, 
Kagoshima University, 21-35 Korimoto 1-chome, Kagoshima 890-0065, Japan}
\email {yokura@sci.kagoshima-u.ac.jp}
\title
{Motivic bivariant characteristic classes}
\thanks {
\\
(1) J\"{o}rg Sch\"{u}rmann supported by the SFB 878 ``groups, geometry and actions''.\\
(2) Shoji Yokura partially supported by JSPS KAKENHI Grant Number 24540085}

\dedicatory{Dedicated to William Fulton, Robert MacPherson\\
and to the memory of Daniel Quillen} 

\begin{document} 
\newtheorem{thm}{Theorem}[section]
\newtheorem{pro}[thm]{Proposition}
\newtheorem{prob}[thm]{Problem}
\newtheorem{cor}[thm]{Corollary}
\newtheorem{con}[thm]{Conjecture}
\newtheorem{lem}[thm]{Lemma}
\theoremstyle{definition}
\newtheorem{ex}[thm]{Example}
\newtheorem{defn}[thm]{Definition}
\newtheorem{rem}[thm]{Remark}
\renewcommand{\rmdefault}{ptm}
\def\alp{\alpha}
\def\be{\beta}
\def\jeden{1\hskip-3.5pt1}
\def\om{\omega}
\def\bigstar{\mathbf{\star}}
\def\ep{\epsilon}
\def\vep{\varepsilon}
\def\Om{\Omega}
\def\la{\lambda}
\def\La{\Lambda}
\def\si{\sigma}
\def\Si{\Sigma}
\def\Cal{\mathcal}
\def\m {\mathcal}
\def\ga{\gamma}
\def\Ga{\Gamma}
\def\de{\delta}
\def\De{\Delta}
\def\bF{\mathbb{F}}
\def\bH{\mathbb H}
\def\bPH{\mathbb {PH}}
\def \bB{\mathbb B}
\def \bA{\mathbb A}
\def \bOB{\mathbb {OB}}
\def \bM{\mathbb M}
\def \bOM{\mathbb {OM}}
\def \calB{\mathcal B}
\def \bK{\mathbb K}
\def \bG{\mathbf G}
\def \bL{\mathbf L}
\def\bN{\mathbb N}
\def\bR{\mathbb R}
\def\bP{\mathbb P}
\def\bZ{\mathbb Z}
\def\bC{\mathbb  C}
\def \bQ{\mathbb Q}
\def\op{\operatorname}

\begin{abstract} Let $K_0(\m V/X)$ be the relative Grothendieck group of varieties over
$X\in Obj(\m V)$, with $\m V=\m V^{(qp)}_k$ (resp. $\m V=\m V^{an}_c$) the category of
(quasi-projective) algebraic (resp. compact complex analytic) varieties over a base field $k$.
Then we constructed  the motivic Hirzebruch class transformation ${T_y}_*: K_0(\m V /X) \to H_*(X) \otimes \bQ[y]$ in the algebraic context for $k$ of characteristic zero, with 
$H_*(X)=CH_*(X)$ (resp. in the complex algebraic or analytic context, with
$H_*(X)=H^{BM}_{2*}(X)$).
It ``unifies" the well-known three characteristic class transformations of singular varieties: MacPherson's Chern class, Baum--Fulton--MacPherson's Todd class  and the $L$-class of Goresky--MacPherson and Cappell--Shaneson. In this paper we construct a bivariant relative Grothendieck group $\bK_0(\m V/X \to Y)$ for $\m V=\m V^{(qp)}_k$ (resp., $\m V^{an}_c$) so that
$\bK_0(\m V/X \to pt)=K_0(\m V/X)$ in the algebraic context with
$k$ of characteristic zero (resp., complex analytic context).

We also construct in the algebraic context (in any characteristic) two Grothendieck transformations $mC_y=\La_y^{mot}: \bK_0(\m V^{qp}/X \to Y) \to \bK_{alg}(X \to Y)\otimes \bZ[y]$ and
$T_y: \bK_0(\m V^{qp}/X \to Y) \to \bH(X \to Y) \otimes \bQ[y]$
with $\bK_{alg}(f)$ the bivariant algebraic $K$-theory of $f$-perfect complexes and $\bH$
the bivariant operational Chow groups (or the even degree bivariant homology in case
$k=\bC$). Evaluating at $y=0$, we get a ``motivic" lift $T_0$ of Fulton--MacPherson's bivariant Riemann--Roch transformation $\tau :\bK_{alg} \to \bH \otimes \bQ$. 
The  covariant transformations $mC_y: \bK_0(\m V^{qp}/X \to pt) \to 
G_0(X)\otimes \bZ[y]$ and
$T_{y*}: \bK_0(\m V^{qp}/X \to pt) \to H_*(X) \otimes \bQ[y]$ agree for $k$ of characteristic zero with our motivic Chern- and Hirzebruch class transformations defined on 
$K_0(\m V^{qp}/X)$.
Finally, evaluating at $y=-1$, for $k$ of characteristic zero we get a ``motivic" lift $T_{-1}$ of Ernstr\"{o}m-Yokura's bivariant Chern class transformation $\gamma: \tilde{\bF}\to CH$.
\end{abstract}

\maketitle

\section{Introduction}\label{intro} 

The classical theory of characteristic classes of vector bundles is a natural transformation from the contravariant monoid functor $(\m Vect,\oplus)$ of isomorphism classes of complex or algebraic vector bundles,
or the associated Grothendieck group $K^0$,  to a contravariant cohomology theory $H^*$. When it comes to characteristic classes of singular spaces, they have been so far formulated as natural transformations from certain covariant theories to a covariant homology theory $H_*$. Topologically or geometrically, the following characteristic classes of singular spaces are most important and have been well-investigated by many people. Here we work either in the category
$\m V=\m V^{(qp)}_k$ of (quasi-projective) algebraic varieties 
(i.e. reduced separated schemes of finite type) over a base field $k$, with $H_*(X)=CH_*(X)$ the Chow homology groups, or in the category $\m V=\m V^{an}_c$ of compact reduced complex analytic spaces, with $H_*(X)=H^{BM}_{2*}(X)$ the even degree Borel-Moore homology
in the complex algebraic or analytic context:

\begin{itemize}
\item MacPherson's Chern class transformation \cite{BSY1, Ken, MacPherson}:
$$c_*: F(X) \to H_*(X),$$ 
defined on the group $F(X)$ of constructible functions in the algebraic context for $k$ of characteritic zero or in the compact complex analytic context.
\item Baum--Fulton--MacPherson's  Todd class or Riemann--Roch transformation \cite{BFM, Fulton-book}:
$$td_*: G_0(X) \to H_*(X)\otimes \bQ,$$
defined on the Grothendieck group $G_0(X)$ of coherent sheaves in the algebraic context
in any characteristic. In the compact complex analytic context such a transformation
can be deduced (compare with \cite{BSY1}) from Levy's $K$-theoretical Riemann-Roch transformation \cite{Levy}.
\item Goresky-- MacPherson's homology $L$-class \cite{GM}, which is extended as a natural transformation by Cappell-Shaneson \cite{CS} (see also \cite{BSY1, Yokura-TAMS, Woolf}):
$$L_*: \Omega_{sd}(X) \to  H_*(X)\otimes \bQ$$
defined on the cobordism group $\Omega(X)$ of selfdual constructible sheaf complexes.
This transformation is only defined for compact spaces in the complex algebraic or analytic context, with $H_*$ the usual homology, since its definition is based on a corresponding signature invariant together
with the Thom-Pontrjagin construction.
\end{itemize}

In 1973 R. MacPherson gave a survey talk about characteristic classes of singular varieties, and his survey article \cite{MacPherson2} ends with the following remark:\\

\emph{``It remains to be seen whether there is a unified theory of characteristic classes of singular varieties like the classical one outlined above."}\footnote{At that time Goresky--MacPherson's homology $L$-class was not available yet and it was defined only after the theory of Intersection Homology \cite{GM} was invented by Mark Goresky and Robert MacPherson. }\\

In our previous paper \cite{BSY1} (see also \cite {BSY2}, \cite{SY}, \cite{Sch-MSRI} and \cite{Yokura-MSRI}) we introduced in the algebraic context for $k$ of characteristic zero, as well as in the compact complex analytic context, the motivic Hirzebruch class transformation
$${T_y}_*: K_0(\m V/X) \to H_*(X)\otimes \bQ[y],$$
defined on  the relative Grothendieck group $K_0(\m V/X)$ of varieties over
$X\in Obj(\m V)$, with $\m V=\m V^{(qp)}_k$ resp. $\m V=\m V^{an}_c$.
This Hirzebruch class transformation ``unifies" the above three characteristic classes $c_*, td_*, L_*$ (see also \S 3) in the sense that we have the following commutative diagrams
of transformations:
 
$$\xymatrix{
& K_0(\Cal V/X)  \ar [dl]_{\epsilon} \ar [dr]^{{T_{-1}}_*} \\
{F(X) } \ar [rr] _{c_*}& &  H_*(X)\otimes \bQ.}
$$

$$\xymatrix{
&  K_0(\Cal V/X)  \ar [dl]_{mC_0} \ar [dr]^{{T_{0}}_*} \\
{G_0(X) } \ar [rr] _{td_*}& &  H_*(X)\otimes \bQ.}
$$

$$\xymatrix{
& K_0(\Cal V/X)  \ar [dl]_{sd} \ar [dr]^{{T_{1}}_*} \\
{\Omega_{sd}(X) } \ar [rr] _{L_*}& &  H_*(X)\otimes \bQ.}
$$

This ``unification" could be considered as a positive answer to the above MacPherson's remark. 
The commutativity of the diagrams above follows (by the functoriality for proper morphisms) already from the normalization condition
$$T_{y*}(X):=T_{y*}([id_{X}])= T^{*}_{y}(TX) \cap [X],$$
for $X$ a smooth manifold, since by ``resolution of singularities" the group $ K_0(\Cal V/X)$ is generated by isomorphism classes $[V \xrightarrow {h} X]$ 
of proper morphisms $h:V \to X$ with $V$ smooth.
Here the Hirzebruch class $T^{*}_{y}(E)$ of the complex or algebraic vector bundle $E$ over $X$ is defined to be (see \cite{Hirzebruch, HBJ}): 
$$T^{*}_{y}(E) := \prod _{i=1}^{\op {rank} E} Q_{y}(\alpha_i)\in H^*(X) \otimes \bQ[y],$$
with 
$$Q_{y}(\alpha):= \frac{\alpha(1+y)}{1-e^{-\alpha(1+y)}} -\alpha y
\quad \in \bQ[y][[\alpha]] \:. $$
Here $\alp _i$'s are the Chern roots of $E$, i.e., $\displaystyle c(E) = \prod_{i=1}^{\op{rank(E)}} (1 + \alp_i).$ Note that $Q_{y}$ is a normalized power series, i.e. $Q_{y}(0)=1$, with:
\begin{itemize}
\item $T^{*}_{-1}(E) =c(E)$ the Chern class, since $Q_{-1}(\alpha)=1+\alpha$.
\item $T^{*}_{0}(E) =td(E)$ the Todd class, since $\displaystyle  Q_{0}(\alpha)=\frac{\alpha}{1-e^{-\alpha}}$.
\item $T^{*}_{1}(E) =L(E)$ the Thom--Hirzebruch $L$-class, since $\displaystyle Q_{1}(\alpha)=\frac{\alpha}{\tanh \alpha}$.
\end{itemize}

Moreover, we also constructed in \cite{BSY1} 
 in the algebraic context for $k$ of characteristic zero, and in the compact complex analytic context, the motivic Chern class transformation
$$mC_y: K_0(\m V/X) \to G_0(X)\otimes \bZ[y].$$
This satisfies the normalization condition 
$$mC_{y}(X):= mC_{y}([id_{X}])= \sum_{i=0}^{dim(X)} \; [\Lambda^{i} T^{*}X]\cdot y^{i}
= \lambda_{y}([T^{*}X])\cap [\m O_{X}]$$
for $X$ a smooth manifold,
with $\lambda_{y}$ the total $\lambda$-class. 
In the compact complex analytic (or complex algebraic) context, the transformation $mC_y$ could also be composed
with the $K$-theoretical {\em Riemann-Roch transformation\/}
$$\alpha:  G_{0}(X)\to K^{top}_{0}(X) $$ 
to the (periodic) topological $K$-homology
(in even degrees) constructed by Levy \cite{Levy} (generalizing the
corresponding transformation of Baum-Fulton-MacPherson \cite{BFM2} for the
quasi-projective complex algebraic context).
Then the Hirzebruch class transformation ${T_y}_*$ could also be defined as the composition 
$td_*\circ mC_y$, 
renormalized by the multiplication $\times (1+y)^{-i}$ on $H_i(X)\otimes \bQ[y]$
to fit with the normalization condition above. So $mC_y$ could be 
considered as a $K$-theoretical refinement of ${T_y}_*$.\\

Note that all the source and target functors appearing above are not only functorial for proper morphisms,
but also have compatible {\em cross products} $\times$ and {\em pullback Gysin homomorphisms} $f^!$ for a smooth morphism $f$.
Moreover, all the characteristic class transformations $c\ell_*$ above (like $c_*, td_*, L_*, mC_y, T_{y*}$) commute with the 
cross products $\times$. Similarly, they commute for a smooth morphism $f$ with the pullback Gysin homomorphisms $f^!$ only up to a correction factor 
$c\ell^*(T_f)$ given by the corresponding cohomological characteristic class of the tangent bundle $T_f$ to the fibers of $f$, i.e. one gets a 
{\em Verdier-Riemann-Roch formula} (see \cite{BSY1}):
$$c\ell_*\circ f^! = c\ell^*(T_f)\cap (f^! \circ c\ell_*)\:.$$
This generalizes a corresponding normalization condition for $X$ a smooth manifold (so that the constant map $X\to pt$ is smooth).
All these properties can be stated in a very efficient way by just saying that $cl_*$ is a natural transformation of {\em Borel-Moore functors} (with product)
in the sense of \cite{LP, Yokura-obt}, if the Gysin maps $f^!$ of the target functors are ``redefined or twisted" by the characteristic class
$c\ell^*(T_f)$ of the tangent bundle $T_f$ to the fibers of $f$ (see \cite{Quillen} and \cite[\S 4.1.9]{LM}). 
Here it is only important that the target functors
of our transformations $cl_*$ have a suitable theory of characteristic classes of (complex or algebraic) vector bundles (like first Chern classes of 
line bundles). So only the target functors should be an {\em oriented Borel-Moore (weak) homology theory} in the sense of Levine-Morel \cite{LM}
(like $CH_*, G_0$),
generalizing, 
in the algebraic context, the notion of a ``complex oriented (co)homology theory" (like $H^{BM}_*, K^{top}_{0}$) 
introduced by Quillen \cite{Quillen} in the context of
differentiable manifolds.
In fact, Quillen \cite{Quillen} introduced in geometric terms {\em complex cobordism} $\Omega_*^U$ as a  universal ``complex oriented (co)homology 
theory". More recently, Levine-Morel \cite{LM} introduced  {\em algebraic cobordism} $\Omega_*^{alg}$ as a universal ``oriented Borel-Moore (weak) 
homology theory" in the algebraic context over a base field of characteristic zero (see also Levine--Pandharipande \cite{LP} for a more geometric approach).\\

In early 1980's William Fulton and Robert MacPherson have introduced the notion of bivariant theory as a {\it categorical framework for the study of singular spaces}, which is the title of their AMS Memoir book \cite{Fulton-MacPherson} (see also Fulton's book
\cite{Fulton-book}). As reviewed very quickly in \S 2, a bivariant theory is definded on morphisms, instead of objects,  and unifies both a covariant functor and a contravariant functor. Important objects to
 be investigated in Bivariant Theories are what they call \emph{Grothendieck transformations} between given two bivariant theories. 
 A Grothendieck transformation is a bivariant version of a natural transformation. A bit more precisely, the main objective of \cite{Fulton-MacPherson} are 
 bivariant-theoretic Riemann--Roch transformations or bivariant analogues of various theorems of 
Grothendieck--Riemann--Roch type and Verdier--Riemann--Roch type (as mentioned before).
 
 A key example of \cite[Part II] {Fulton-MacPherson} is the bivariant Riemann--Roch transformation $\tau :\bK_{alg} \to \bH \otimes \bQ$ on the category $\m V=\m V^{qp}_{\bC}$ of complex quasi-projective varieties,
with $\bK_{alg}(f)$ the bivariant algebraic $K$-theory of $f$-perfect complexes and $\bH$
the even degree bivariant homology. It unifies the covariant Todd class transformation $td_*$
and the contravariant Chern character $ch$. 
An algebraic version on the category
$\m V=\m V^{qp}_{k}$ of quasi-projective varieties over a base field $k$ of any characteristic
was constructed later on in \cite[Example 18.3.19]{Fulton-book}, with  $\bH=CH$
the bivariant operational Chow groups. 
As another example, Fulton-MacPherson constructed in \cite[Part II]{Fulton-MacPherson} in the complex quasi-projective context
also a Grothendieck transformation $\alpha: \bK_{alg} \to \bK_{top}$ between their bivariant algebraic and topological $K$-theory, as well as in
\cite[Part I, \S 6] {Fulton-MacPherson} a bivariant Whitney class transformation. And they asked
in the complex algebraic context for a corresponding bivariant Chern class transformation 
$\gamma: \bF\to \bH$ on their bivariant theory $\bF$ of constructible functions satisfying a suitable local Euler condition, which generalizes the covariant MacPherson Chern class transformation $c_*$. For $\bH$ the even degree bivariant homology, this problem was solved by Brasselet \cite{Brasselet} in a suitable
context (even for compact analytic spaces), 
whereas Ernstr\"{o}m-Yokura \cite{EY1} solved it for
$\bH=A^{PI} (\supset CH)$ another bivariant operational Chow group theory (for the notation $A^{PI}$ see \cite{EY1}). 
Finally, relaxing the local Euler condition, they introduced in \cite{EY2} a bivariant Chern class transformation $\gamma: \tilde{\bF}\to CH$ from another bivariant theory $\tilde{\bF}$ of constructible functions.
This last approach is based on the usual calculus of constructible functions and the surjectivity of $c_*: F(X)\to CH_*(X)$, so it works in the algebraic context over any base field $k$ of characteristic zero (even though it was stated
in \cite{EY2} only in the complex algebraic context). Here $\tilde{\bF}(X\to pt)=F(X)$ follows from the multiplicativity of $c_*$ with respect to cross products $\times$.\\

One of the main objects of the present paper is to obtain two bivariant analogues 
$$mC_y=\La_y^{mot}: \bK_0(\m V^{qp}/X \to Y) \to \bK_{alg}(X \to Y)\otimes \bZ[y]$$ and
$$T_y: \bK_0(\m V^{qp}/X \to Y) \to \bH(X \to Y) \otimes \bQ[y]$$
of the motivic Chern and Hirzebruch class transformations $mC_y$ and ${T_y}_*$, with 
$T_y$  defined as the composition 
$\tau \circ mC_y$, 
renormalized by the multiplication $\times  (1+y)^i$ on $\bH^i(-)\otimes \bQ[y]$.
Moreover,
$T_y$ unifies the  bivariant Riemann--Roch transformation $\tau :\bK_{alg} \to \bH \otimes \bQ$
 (for $y=0$) and the
bivariant Chern class transformation $\gamma: \tilde{\bF}\to CH$ (for $y=-1$).
Note that a bivariant $L$-class transformation (corresponding to $y=1$) is still missing.
In \cite{BSY3, BSY4} we considered a kind of general construction of a bivariant analogue of a given natural transformation between two covariant functors, but our approach presented in this paper is quite different from it. The former is more ``operational", but the latter  is more ``direct" and very ``motivic", as outlined below.\\

Let $\m V=\m V^{(qp)}_k$ be the category of (quasi-projective) algebraic varieties 
(i.e. reduced separated schemes of finite type) over a base field $k$ of any characteristic, or let $\m V=\m V^{an}_c$ be the category of compact reduced complex analytic spaces.
 On the category $\m V$ we define 
$$\bM(\m V/X \xrightarrow {f} Y)$$ 
to be the free abelian group on the set of isomorphism classes $[V \xrightarrow {h} X]$ of proper morphisms $h:V \to X$ such that the composite 
$f \circ h: V \to Y$ is a smooth morphism, in other words, $h: V \to X$ is ``a left quotient" of a smooth morphism $s: V \to Y$ devided 
by the given morphism $f$:
$$ f \circ h = s \quad \text {or} \quad  h = \frac {s}{f},$$
$$\xymatrix{
& V\ar [dl]_{h} \ar [dr]^{s} \\
X\ar [rr] _{f}& &  Y.}$$
Here two morphisms $h: V \to X$ and $h': V' \to X$ are called isomorphic to each other if there exists an isomorphism $\phi: V \xrightarrow {\cong} V'$ such that the following diagram commutes
$$\xymatrix{
V\ar[dr]_ {h}\ar[rr]^ {\phi} && V' \ar[dl]^{h'}\\
& X.}$$

\begin{thm}
The association $\bM(\m V/-)$ becomes a bivariant theory with natural bivariant-theoretic operations. 
\end{thm}

 \begin{rem}
 The associated ``cohomology theory"  ${\bM}^*(\m V/X) = \bM(\m V/X \xrightarrow{\op {id}_X}  X)$ is the free abelian group generated by the isomorphism 
 classes of proper and smooth morphism $[V \xrightarrow{h}  X]$. So it is a geometric approach to cohomology classes in the algebraic or compact complex 
 analytic context, based on proper submersions having a tangent bundle to the fibers (as a substitute for a bundle theoretic approach to cohomology 
 classes in topology). Moreover, the bivariant theory $\bM(\m V/-)$ based on (isomorphism classes of) 
 ``left quotients" $h = \frac {s}{f}$ with $h$ proper and $s$ smooth fits nicely with the recent approach of Emerson-Meyer \cite{EM2, EM3}
 to ``(bivariant) $KK$-theory via correspondences" (here ``bivariant" has a meaning different from the notion of Fulton-MacPherson used in this paper).
 In fact, one can see the ``left quotient" $h = \frac {s}{f}$ also as a correspondence between $X$ and $Y$ fitting with the given morphism
 $f: X\to Y$. Forgetting $f$, one can define the free abelian group $\bM(\m V/X,Y)$ generated by the isomorphism 
 classes of such correspondences (with $h$ proper and $s$ smooth), with the ``usual" composition $\circ$ of correspondences. Then
 our definition of the bivariant product $\bullet$ fits under the tautological map
 (forgetting $f$):
 $$forget: (\bM(\m V/X \xrightarrow {f} Y),\bullet) \to (\bM(\m V/X,Y),\circ)$$
 with the composition product of these correspondences (and it is also functorial in $X$ with respect to the corresponding pushforwards under proper 
 morphisms). As will be explained elsewhere (see \cite{BaSY}), in the context of complex varieties there is also a similar transformation
 $$(\bM(\m V/X,Y),\circ) \to (KK(X,Y),\circ)  $$
 to the ``$KK$-theory via correspondences"  of Emerson-Meyer \cite{EM2, EM3} 
 (and more generally to their counterpart based on a complex oriented cohomology theory).
 \end{rem}

 Let $\bB$ be a  bivariant theory on  $\m V$ such that a smooth morphism $f: X\to Y$ has a 
 {\em stable orientation} $\theta(f)\in \bB(f)$, like $\bM(\m V/-)$, with
 $\theta(f):=[X \xrightarrow {\op {id}_X} X]$ (these notions will be explained in \S 2).
In the algebraic context, examples for $\bB$ are given by the bivariant 
algebraic $K$-theory $\bK_{alg}$ of relative perfect complexes and the 
bivariant operational Chow groups $CH$. Examples in the complex algebraic or analytic context are given by the (even degree) bivariant topological $K$-theory $\bK^{top}$ or homology theory $\bH\otimes R$ of Fulton--MacPherson \cite{Fulton-MacPherson}, with $R=\bZ, \bQ, \bQ[y]$.
Another example is Fulton-MacPherson's bivariant theory $\bF$ of constructible functions in the complex algebraic or analytic context,
or Ernstr\"{o}m-Yokura's bivariant theory $\tilde{\bF}$ of constructible functions in the algebraic context over a base field of characteristic zero,
with $\theta(f)=\jeden_f:=1_X$ for a smooth morphism $f: X\to Y$.

\begin{thm}\label{thm:univ} Let $\bB$ be a  bivariant theory on  $\m V$ such that a smooth morphism $f: X\to Y$ has a 
stable orientation $\theta(f)\in \bB(f)$.
Then there exists a unique Grothendieck transformation
$$\ga:=\ga_{\theta}: \bM(\m V/-) \to \bB(-)$$
satisfying the normalization condition that for a smooth morphism $f:X \to Y$ the following identity holds in $\bB(X \xrightarrow {f} Y)$:
$$\ga([X \xrightarrow {\op {id}_X} X]) =  \theta(f).$$
\end{thm}

\begin{cor}
Let $c\ell: Vect(-) \to  \bB^*(-)$ be a contravariant functorial characteristic  class of algebraic (or analytic) vector bundles
with values in the associated cohomology theory, which is multiplicative in the sense that $c\ell(V) = c\ell(V') c\ell(V'')$ for 
any short exact sequence of vector bundles $0\to V'\to V \to V'' \to 0$, with $c\ell(T_{pt})=1_{pt}\in \bB^*(\{pt\})$. 
Assume that $c\ell$ commutes with the stable orientation $\theta$, i.e.
$$\theta(f)\bullet cl(V)=f^*cl(V)\bullet \theta(f)$$ 
for all smooth morphism $f: X\to Y$ and $V\in Vect(Y)$.
Then there exists a unique Grothendieck transformation
$$\ga_{c\ell}: \bM(\m V/-) \to \bB(-)$$
satisfying the normalization condition that for a smooth morphism $f:X \to Y$ the following identity holds in $\bB(X \xrightarrow {f} Y)$:
$$\ga_{c\ell}([X \xrightarrow {\op {id}_X} X]) = c\ell(T_f) \bullet \theta(f).$$
Here $T_f$ is the relative tangent bundle  of the smooth morphism $f$.
\end{cor}
This follows from Theorem \ref{thm:univ} by using the new ``twisted" stable orientation $\theta'(f):=c\ell(T_f) \bullet \theta(f)$ for a smooth morphism 
$f: X\to Y$. Similar twisting constructions are due to Quillen \cite{Quillen} (resp., Levine-Morel \cite[\S 4.1.9]{LM}) in the context of
complex oriented (co)homology theories (resp., oriented Borel-Moore (weak) homology theories).\\

 This $\ga_{c\ell}: \bM(\m V/ -) \to \bB( -)$ should be considered as \emph{a ``pre-motivic" bivariant theory of characteristic classes}.
In particular, if we consider the case of a mapping $X \to pt$ to a point, 
${\bM}_ *(\m V/X) := \bM(\m V/X \to pt)$ behaves covariantly for proper morphisms and we have
\begin{cor} ${\ga_{c\ell}}_*: {\bM}_ *(\m V/-)\to \bB_*(-)$
is a unique natural transformation satisfying the ``normalization condition" that for a smooth variety $X$ 
$${\ga_{c\ell}}_*([X \xrightarrow {\op {id}_X} X]) = c\ell(TX) \cap [X],$$
with $[X]:= \theta(p)\in \bB_*(X)$ (resp., $[X]:= p^!(1_{pt})\in \bB_*(X)$) the ``fundamental class" of $X$ given by the canonical orientation
(resp., the Gysin homomorphism) of the smooth morphism $p: X\to pt$.
\end{cor}

\begin{rem}\label{univ-BM} We note that in fact here $\bB_*(-)$ does not need to be associated to a bivariant theory, e.g. it would be enough that
$\bB_*(-)$ is an  oriented Borel-Moore (weak) homology theory like $\Omega_*^{alg}$ (or a complex oriented (co)homology theory
like $\Omega^{U}_*$).
In fact ${\bM}_ *(\m V/-)$ is a {\em universal Borel Moore functor} with product (,but without an additivity property), see \cite{Yokura-obt}.
Also the characteristic class $c\ell$ does not need to be multiplicative for the definition of the natural transformation 
${\ga_{c\ell}}_* :{\bM}_ *(\m V/-)\to \bB_*(-)$, although we do need the multiplicativity of $c\ell$ for the  multiplicativity
of ${\ga_{c\ell}}_*$ with respect to cross products $\times$. Similarly, for a corresponding Verdier-Riemann-Roch formula, we need the
compability
$$f^!(c\ell(V)\cap -) = f^*c\ell(V)\cap f^!(-)$$
of $c\ell$ with the Gysin homomorphism $f^!$ for a smooth morphism $f: X\to Y$ and $V\in Vect(Y)$.
\end{rem}
${\ga_{c\ell}}_*: {\bM}_ *(\m V/-)\to \bB_*(-)$ should be considered as a \emph{``pre-motivic" characteristic class transformation of possibly 
singular varieties}, e.g. like $\ga_*: {\bM}_ *(\m V/-)\to \Omega_*^{alg}$ resp. $\ga_*: {\bM}_ *(\m V/-)\to  \Omega^{U}_*$ associated to $c\ell(V):=1_Y$ 
for  $V\in Vect(Y)$. These fit, in the complex algebraic context, into the following commutative diagram of transformations:
\begin{equation}\label{diagram}
 \begin{CD}
  {\bM}_ *(\m V/X) @> \ga_* >> \Omega_*^{alg}(X) @>>> \Omega^{U}_*(X) \\
  @VVV @VVV @VVV \\
  K_0(\m V/X) @> mC_y >> G_0(X)[y] @> \alpha >> K_0^{top}(X)[y] \\
   @|  @VVV @VVV \\
   K_0(\m V/X) @> T_{y*} >> CH_*(X)\otimes \bQ[y] @>>> H^{BM}_{2*}(X)\otimes \bQ[y]\\
   @V \epsilon VV  @V y=-1 VV @V y=-1 VV \\
   F(X)  @> c_* >> CH_*(X)\otimes \bQ @>>> H^{BM}_{2*}(X)\otimes \bQ \:.
 \end{CD}
\end{equation}
The left (resp. outer) part of this diagram is also available in the algebraic context over a base field of characteristic zero
(resp. in the compact complex analytic context).

\begin{rem} The horizontal transformations in the upper line are the canonical ones associated to different universal theories, with
 ${\bM}_ *(\m V/-)$ the universal Borel Moore functor with product (but without an additivity property),  $\Omega_*^{alg}$ the
 universal oriented Borel-Moore (weak) homology theory and $\Omega^{U}_*$ the universal complex oriented (co)homology theory.
 Similarly, the theories $H_*$ in the last two vertical lines represent different such homology theories (with a $\cap$-product action of characteristic classes
 of vector bundles) in the algebraic resp. topological context,
 like the universal theories $\Omega_*^{alg}, \Omega^{U}_*$, the $K$-theoretical theories $G_0, K_0^{top}$ or the classical theories $CH_*, H^{BM}_{2*}$. 
 Also these six homology theories are associated to suitable bivariant theories, which are due to 
 Fulton-MacPherson \cite{Fulton-MacPherson}, except for algebraic cobordism $\Omega_*^{alg}$, where
 a corresponding ``operational" bivariant version has been recently constructed 
 by Gonz\'alez and Karu \cite{GK}.

In the topological context one also  has 
 {\em Mayer-Vietoris and long exact homology sequences},
 whereas in the algebraic context one has {\em short exact sequences}
\begin{equation}\label{short}
 \begin{CD}
   H_*(Z) @> i_* >> H_*(X) @> j^* >> H_*(U)@>>> 0 
   \end{CD}\end{equation}
for $i: Z\to X$ the inclusion of a closed algebraic subset, with open complement $j: U:=X\backslash Z\to X$.
For our unification, it is important to work with more general theories like ${\bM}_ *(\m V/-)$ and $K_0(\m V/-)$,
which are {\em not} oriented Borel-Moore (weak) homology (or complex oriented (co)homology) theories, like the 
group $F(X)$ of constructible functions in relation to MacPherson's Chern class transformation.
Here we do not have such a short exact sequence (\ref{short}) for ${\bM}_ *(\m V/-)$, but in the case of 
$K_0(\m V/-)$ (and also for $F(-)$) we even have short exact sequences
$$
  \begin{CD}
0@>>> K_0(\m V/Z) @> i_* >> K_0(\m V/X) @> j^* >> K_0(\m V/U)@>>> 0 \:.
   \end{CD}$$
But another important property, which fails for them, is ``homotopy invariance", e.g. 
$$p^*: K_0(\m V/X) \to K_0(\m V/X\times \bA^1)$$
is injective but not surjective for the projection $p: X\times \bA^1\to X$ (and similarly for $F(-)$).
\end{rem}

A true \emph{``motivic" characteristic class transformation of possibly singular varieties}
should factorize as in (\ref{diagram})  over the canonical group homomorphism
\begin{equation}
q: {\bM}_*(\m V/X) \to K_0(\m V/X),
\end{equation}
like the transformations ${\ga_{c\ell}}_*$ associated to the multiplicative characteristic classes $c\ell$ given by $c, td, L, T^*_y$, 
or the total lambda-class $\lambda_y((-)^*)$ of the dual vector bundle, as mentioned before
(in the complex analytic or algebraic context over a base field of characteristic zero).

Only then we can also speak of the corresponding characteristic class 
$$cl_*(X) := {\ga_{c\ell}}_*([\op{id}_X])$$
of a \emph{singular} space $X$, where 
${\ga_{c\ell}}_*$ is the bottom homomorphism in the following diagram:
 \begin{equation}\label{diagram2}
\xymatrix{
& {\bM}_*(\m V/-) \ar [dl]_{q} \ar [dr]^{{\ga_{c\ell}}_*} \\
{K_0(\m V/-)} \ar [rr] _{{\ga_{c\ell}}_*}& &  \bB_*(-)\,.}
\end{equation}
Note that for a singular space $X$ one has the distinguished element $[X \xrightarrow {\op {id}_X} X] \in K_0(\m V/X)$, 
but $[X \xrightarrow {\op {id}_X} X]$ cannot be defined in ${\bM}_*(\m V/X)$.

\begin{rem} In fact in \cite{BSY1} we proved more in the complex analytic or algebraic context over a base field of characteristic zero, with $\bB=CH\otimes R$ or $\bB=\bH\otimes R$:
The induced genus ${\ga_{c\ell}}_*: \bM(\m V/ pt) \to H_*(pt)\otimes R=R$
of a corresponding multiplicative characteristic class $c\ell$ has to be a specialization
of the Hirzebruch $\chi_y$-genus characterized by
$$\chi_y(\mathbb {P}^n)=1 - y + y^2 + \cdots +(-y)^n.$$
Moreover,
the Hirzebruch class $T^*_y$ is for $R=\bQ[y]$ the only multiplicative characteristic class $c\ell$ with this property, which is defined by a normalized power series in $\bQ[y][[\alpha]]$.
So it is the only such characteristic class $c\ell$, 
for which ${\ga_{c\ell}}_*: \bM(\m V/X \to pt) \to H_*(X)\otimes \bQ[y]$ can be factorized over the motivic group $K_0(\m V/X)$:
\begin{equation}\label{diagram2b}
\xymatrix{
& {\bM}_*(\m V/X) \ar [dl]_{q} \ar [dr]^{{\ga_{c\ell}}_*} \\
{K_0(\m V/X)} \ar [rr] _{{T_y}_*}& &  H_*(X)\otimes \bQ[y]\,.}
\end{equation}
\end{rem}
 
By ``resolution of singularities", the canonical group homomorphism
$q: {\bM}_*(\m V/X) \to K_0(\m V/X)$ is surjective
in the complex analytic or algebraic context over a base field of characteristic zero.
Moreover, using the ``weak factorization theorem" of \cite{AKMW, W}, its kernel was described by Bittner \cite{Bittner} in terms of a 
``blow-up relation". In some sense (as mentioned by a referee), this can be seen as a counterpart of the ``Conner-Floyd theorem" \cite{CF} in topology 
(or 
\cite{LM} in algebraic geometry), about recovering $K$-theory from cobordism.
Here we introduce the following bivariant analogue of the ``blow-up relation":

\begin{defn}\label{BL} For a morphism $f: X \to Y$ in the category $\m V=\m V^{(qp)}_k$ or 
$\m V=\m V^{an}_c$, we consider a blow-up diagram
$$\begin{CD}
E @> i'>> Bl_{S}X' \\
@VV q' V  @VV q V \\
S @> i >> X' @> h >> X  @> f >> Y\:,
\end{CD} $$
with $h$ proper and $i$ a closed embedding such that
$f \circ h$ as well as $f \circ h \circ i$ are smooth.
Here $q: Bl_{S}X' \to X'$ is the blow-up of $X'$ along $S$, with $q':E \to S$ the exceptional divisor map. Then also $f \circ h \circ q$ and $f \circ h \circ i\circ q'$ are smooth (with $Bl_{S}X'$ and $E$ quasi-projective in the case $\m V=\m V^{qp}_k$). 
Let $\mathbb {BL}(\m V/ X \xrightarrow{f}  Y)$ be the free abelian subgroup of 
$\bM(\m V/X \xrightarrow{f}  Y)$ generated by 
$$ [Bl_{S}X' \xrightarrow{h q}  X] - [E \xrightarrow{hiq'}  X] -  [X'\xrightarrow{h}  X] + [S\xrightarrow{hi}  X]$$
for any such diagram, and define
$$\bK_0(\m V/X \xrightarrow{f}  Y) := \frac{\bM(\m V/X \xrightarrow{f}  Y)}
{\mathbb {BL}(\m V/X \xrightarrow{f}  Y)}.$$
The corresponding equivalence class of $[V \xrightarrow{p}  X]$ shall be denoted by $\Bigl[[V \xrightarrow{p}  X] \Bigr].$
\end{defn}
 
Note that for $Y=pt$ a point, the smoothness of $f \circ h$ and $f \circ h \circ i$ above
is equivalent to $X'$ and $S$ are smooth manifolds. So in this case 
$\mathbb {BL}(\m V/ X \to pt)$ reduces to the ``blow-up relation" considered by Bittner.
In particular, we get a canonical group homomorphism $\bK_0(\m V/X\to pt) \to K_0(\m V/X)$
to the relative motivic Grothendieck group of varieties over $X$, which by
Bittner's theorem is an isomorphism in the complex analytic or algebraic context over a base field of characteristic zero.

\begin{thm}\label{thm:main} Let $\m V=\m V^{(qp)}_k$ be the category of (quasi-projective) algebraic varieties 
(i.e. reduced separated schemes of finite type) over a base field $k$ of any characteristic, or let $\m V=\m V^{an}_c$ be the category of compact reduced complex analytic spaces.
\begin{enumerate}
\item[(i)] $\bK_0(\m V / - )$ can be given uniquely the structure of a bivariant theory so that the canonical projection 
$\bB q: \bM(\m V/-) \to \bK_0(\m V / - )$ is a Grothendieck transformation.
\item[(ii)]  There exists a unique Grothendieck transformation
$$mC_y=\La_y^{mot}: \bK_0(\m V_k^{qp}/ - ) \to \bK_{alg}( - )\otimes \bZ[y]$$
satisfying the normalization condition that for a smooth morphism $f: X \to Y$ 
 the following equality holds in $\bK_{alg}(X \xrightarrow {f} Y) \otimes \bZ[y]$:
$$\La_y^{mot}\Bigl(\Bigl[[X  \xrightarrow{\op {id}_X}  X]\Bigr]\Bigr) = \La_y(T^*_f) \bullet \theta(f).$$
\item[(iii)]  Let $T_y : \bK_0(\m V_k^{qp} / - ) \to \bH(-) \otimes \bQ[y]$
be defined as the composition 
$\tau \circ \La_y^{mot}$, renormalized by $\cdot (1+y)^i$ on $\bH^i(-)\otimes \bQ[y]$.
Here $\bH$ is either the operational bivariant Chow group,
or the even degree bivariant homology theory for $k=\bC$, with $\tau$ the corresponding
Riemann-Roch transformation. \\
Then $T_y$ 
is the unique Grothendieck transformation
satisfying the normalization condition that for a smooth morphism $f: X \to Y$  the following equality holds in $\bH(X \xrightarrow {f} Y) \otimes \bQ[y]$:
$$T_y\Bigl(\Bigl[[X  \xrightarrow{\op {id}_X}  X]\Bigr]\Bigr) = T^*_y(T_f) \bullet \theta(f).$$
\end{enumerate}
\end{thm}

\begin{cor}\label{cor-Grothendieck} We have the following commutative diagrams of Grothendieck transformations:
\begin{enumerate}
\item[(i)] $$\xymatrix{
&   \bK_0(\m V_k^{qp} / - )  \ar [dl]_{mC_0} \ar [dr]^{{T_{0}}} \\
{\bK_{alg}( - ) } \ar [rr] _{\tau}& &  \bH(-) \otimes \bQ.}$$
\item[(ii)] $$\xymatrix{
&   \bK_0(\m V_k^{qp} / - )  \ar [dl]_{\epsilon} \ar [dr]^{{T_{-1}}} \\
{\tilde{\bF}( - ) } \ar [rr] _{\gamma}& &  CH(-) \otimes \bQ,}$$
if $k$ is of characteristic zero. Here $\epsilon$ is the unique Grothendieck transformation
satisfying the normalization condition $\epsilon\Bigl(\Bigl[[X  \xrightarrow{\op {id}_X}  X]\Bigr]\Bigr)=\jeden_f$ for a smooth morphism $f: X \to Y$. And similarly for the bivariant Chern class transformation
$\gamma: \bF( - ) \to A^{PI}( - )\otimes \bQ \supset  CH(-) \otimes \bQ$ in case $k=\bC$.
\item[(iii)] Assume $k$ is of characteristic zero. Then the associated covariant transformations in Theorem \ref{thm:main} (ii) and (iii) agree 
under the identification 
$$\bK_0(\m V_k^{qp}/X\to pt) \simeq K_0(\m V_k^{qp}/X)$$
with  the motivic Chern and Hirzebruch class transformations $mC_y$ and ${T_y}_*$.
\end{enumerate}
\end{cor}

Let us finish this introduction with some problems left open:

\begin{enumerate}
\item Our construction of the Grothendieck transformation 
$$mC_y=\La_y^{mot}: \bK_0(\m V_k^{qp}/ - ) \to \bK_{alg}( - )\otimes \bZ[y]$$
based on \cite[Chapter IV, Theorem 1.2.1 and (1.2.6)]{Gros} also works in the algebraic context without considering only quasi-projective varieties,
if one uses the more sophisticated definition of $\bK_{alg}(X \xrightarrow {f} Y)=K_0(D^b_{f-perf}(X))$ as the Grothendieck goup of the triangulated category of $f$-perfect complexes. And a similar definition can also be used in the context of compact complex analytic varieties (cf.  \cite[Part I, \S 10.10]{Fulton-MacPherson} and \cite{Levy2}). Then it seems reasonable, that one can also construct in a similar way in this compact complex analytic context
the Grothendieck transformation $mC_y=\La_y^{mot}$.
Here it would be enough to prove the analogues of \cite[Chapter IV, Theorem 1.2.1 and (1.2.6)] {Gros} in the complex analytic context.

\item Similarly one would like to further construct in this compact complex analytic context also the Grothendieck transformation
$T_y$ based on Levy's  K-theoretical Riemann-Roch transformation $\alpha : \bK_{alg}(-)\to \bK^{top}_0(-)$ from algebraic to topological
bivariant K-theory (see  \cite{Levy2}).
 A key result missing so far is the counterpart 
 $$\alpha(\m O_f)=\theta(f)$$ 
 of \cite[Part II, Theorem 1.4 (3)]{Fulton-MacPherson},
that $\alpha$ identifies for a smooth morphism $f: X\to Y$ the orientation $\m O_f:=[\m O_X]\in \bK_{alg}(X \xrightarrow {f} Y)$ with the
 orientation  $\theta(f)\in \bK^{top}_0(X \xrightarrow {f} Y).$

\item We do not know if Brasselet's bivariant Chern class transformation $\ga: \bF(-) \to \bH(-)$ (see \cite{Brasselet})
satisfies for a smooth morphism $f:X \to Y$ the ``strong normalization condition'' 
$$\ga(\jeden_f) = c(T_f) \bullet \theta(f) \in \bH(X  \xrightarrow{f}  Y).$$
Then Corollary \ref{cor-Grothendieck} (ii) would also be true for Brasselet's bivariant Chern class transformation $\ga: \bF(-) \to \bH(-)$.
\item In a future work we will construct in the compact complex algebraic or analytic context a bivariant analogue $\bB\Omega(-)$
of the cobordism group $\Omega(-)$ of selfdual constructible sheaf complexes, together with a Grothendieck transformation
$sd: \bK_0(\m V/ - ) \to \bB\Omega(-)$. This will be based on suitable Witt-groups of constructible sheaves and some other related topics different from the theme of the present paper. But what is still missing to get the counterpart of Corollary \ref{cor-Grothendieck} (i) and (ii) for $y=1$ is a bivariant
 $L$-class transformation $\bB L: \bB\Omega(-) \to \bH(-)\otimes \bQ$.
\end{enumerate}
 
\begin{rem} Our 
debt to the works of Quillen and Fulton-MacPherson 
should be clear after reading this introduction.
In this paper we focus in the last sections on the unification of different bivariant theories of characteristic classes of singular spaces,
by dividing out our universal bivariant theory by a ``bivariant blow-up relation". 
But similar ideas (with other bivariant relations) should also work for other applications, e.g. in the algebraic geometric context for the construction of a ``geometric" 
bivariant-theoretic version of Levine--Morel's algebraic cobordism (different from the ``operational" vivariant theory of \cite{GK}, e.g. see \cite{Yokura-obt, SY2}).
Similarly, in \cite{BaSY} we will construct in the context of reduced differentiable spaces  a ``geometric" 
bivariant-theoretic version of Quillen's complex cobordism (different from the abstract definition given by the general theory of 
Fulton-MacPherson \cite{Fulton-MacPherson}),
and closely related to the approach of  Emerson-Meyer \cite{EM2, EM3}
 to ``(bivariant) $KK$-theory via correspondences". The corresponding cohomology theory for smooth manifolds will be different and a refinement
 of Quillen's geometric theory of complex cobordism \cite{Quillen}.
\end{rem}


\section {Fulton--MacPherson's bivariant theory}\label{FM-BT}

For the sake of the reader we quickly recall some basic ingredients of Fulton--MacPher- son's bivariant theory \cite{Fulton-MacPherson}. \\

Let $\Cal V$ be a category which has a final object $pt$ and on which the fiber product or fiber square is well-defined, e.g. the category
$\m V^{(qp)}_k$ of (quasi-projective) algebraic varieties 
(i.e. reduced separated schemes of finite type) over a base field $k$, or  $\m V^{an}_{(c)}$ the category of (compact) reduced complex analytic spaces.
 We also consider a class of maps, called ``confined maps" (e.g., proper maps in this algebraic or analytic geometric context), which are closed under composition and base change and contain all the identity maps. Finally, one fixes a class of fiber squares, called ``independent squares" (or ``confined squares", e.g., ``Tor-independent" in algebraic geometry, a fiber square with some extra conditions required on morphisms of the square), which satisfy the following properties:

(i) if the two inside squares in  

$$\CD
X''@> {h'} >> X' @> {g'} >> X \\
@VV {f''}V @VV {f'}V @VV {f}V\\
Y''@>> {h} > Y' @>> {g} > Y \endCD
$$

or

$$\CD
X' @>> {h''} > X \\
@V {f'}VV @VV {f}V\\
Y' @>> {h'} > Y \\
@V {g'}VV @VV {g}V \\
Z'  @>> {h} > Z \endCD
$$

are independent, then the outside square is also independent.

(ii) any square of the following forms are independent:
$$
\xymatrix{X \ar[d]_{f} \ar[r]^{\op {id}_X}&  X \ar[d]^f & & X \ar[d]_{\op {id}_X} \ar[r]^f & Y \ar[d]^{\op {id}_Y} \\
Y \ar[r]_{\op {id}_X}  & Y && X \ar[r]_f & Y}
$$
where $f:X \to Y$ is any morphism. \\

A bivariant theory $\bB$ on a category $\Cal V$ with values in the category of (graded) abelian groups is an assignment to each morphism
$$ X  \xrightarrow{f} Y$$
in the category $\Cal V$ a (graded) abelian group (in most cases we can ignore a possible grading)

$$\bB(X  \xrightarrow{f} Y)$$
which is equipped with the following three basic operations. The $i$-th component of $\bB(X  \xrightarrow{f} Y)$, $i \in \bZ$, is denoted by $\bB^i(X  \xrightarrow{f} Y)$
(with $\bB(X  \xrightarrow{f} Y)=:\bB^0(X  \xrightarrow{f} Y)$ in the ungraded context).\\

{\bf Product operations}: For morphisms $f: X \to Y$ and $g: Y
\to Z$, the ($\bZ$-bilinear) product operation
$$\bullet: \bB^i( X  \xrightarrow{f}  Y) \otimes \bB^j( Y  \xrightarrow{g}  Z) \to
\bB^{i+j}( X  \xrightarrow{gf}  Z)$$
is  defined.

{\bf Pushforward operations}: For morphisms $f: X \to Y$
and $g: Y \to Z$ with $f$ \emph {confined}, the ($\bZ$-linear) pushforward operation
$$f_*: \bB^i( X  \xrightarrow{gf} Z) \to \bB^i( Y  \xrightarrow{g}  Z) $$
is  defined.

{\bf Pullback operations}: For an \emph{independent} square
$$\CD
X' @> g' >> X \\
@V f' VV @VV f V\\
Y' @>> g > Y, \endCD
$$

the ($\bZ$-linear) pullback operation
$$g^* : \bB^i( X  \xrightarrow{f} Y) \to \bB^i( X'  \xrightarrow{f'} Y') $$
is  defined.\\

And these three operations are required to satisfy the seven compatibility axioms (see \cite [Part I, \S 2.2]{Fulton-MacPherson} for details):
\begin{enumerate}
\item[(B-1)] product is associative, 

\item[(B-2)] pushforward is functorial,

\item[(B-3)]  pullback is functorial, 

\item[(B-4)]  product and pushforward commute,

\item[(B-5)]  product and pullback commute, 

\item[(B-6)]  pushforward and pullback commute, and 

\item[(B-7)] projection formula. \\
\end{enumerate}

We also assume that $\bB$ has \emph{units}, i.e., there is an element $1_X \in \bB^0( X  \xrightarrow{\op {id}_X} X)$ such that $\alp \bullet 1_X = \alp$ for all morphisms $W \to X$ and $\alp \in \bB(W \to X)$; such that $1_X \bullet \beta = \beta $ for all morphisms $X \to Y$ and $\beta \in \bB(X \to Y)$; and such that $g^*1_X = 1_{X'}$ for all $g: X' \to X$. \\

Let $\bB, \bB'$ be two bivariant theories on the category $\Cal V$. Then
a {\it Grothendieck transformation} from $\bB$ to $\bB'$
$$\ga : \bB \to \bB'$$
is a collection of group homomorphisms
$$\bB(X \to Y) \to \bB'(X \to Y) $$
for all morphisms $X \to Y$ in the category $\Cal V$, which preserves the above three basic operations (as well as the units, but not necessarily possible gradings): 
\begin{enumerate}
\item[(i)] \quad $\ga (\alp \bullet_{\bB} \be) = \ga (\alp) \bullet _{\bB'} \ga (\be)$, 

\item[(ii)] \quad $\ga(f_{*}\alp) = f_*\ga (\alp)$, and 

\item[(iii)] \quad $\ga (g^* \alp) = g^* \ga (\alp)$. \\
\end{enumerate}
 
Most of our bivariant theories in this paper are \emph{commutative}
(see \cite [\S 2.2]{Fulton-MacPherson}), i.e., if whenever both

$$
\xymatrix{W \ar[d]_{f'} \ar[r]^{g'}&  X \ar[d]^f && W \ar[d]_{g'} \ar[r]^{f'}& Y \ar[d]^{g} \\
Y \ar[r]_{g}  & Z && X \ar[r]_f & Z}
$$
are independent squares, then for 
$\alp \in \bB(X  \xrightarrow {f} Z)$ and $\be \in \bB(Y  \xrightarrow {g} Z)$
$$g^*(\alp) \bullet \be = f^*(\be) \bullet \alp.$$

This is for example the case for all bivariant theories mentioned in the introduction in the algebraic or analytic geometric context, except for the bivariant operational Chow group $CH$, with bivariant algebraic K-theory
$\bK_{alg}$ and bivariant constructible functions $\bF, \tilde{\bF}$ examples of ungraded theories. Here $CH$ is at least commutative in the context of a base field $k$ of characteristic
zero, by \cite[Example 17.4.4] {Fulton-book} (using resolution of singularities).
Similarly the bivariant homology $\bH$ is commutative,
if we restrict ourselves to the even degree part only 
(otherwise $\bH$ would be \emph{skew-commutative}, i.e. $g^*(\alp) \bullet \be = (-1)^{\op {deg}(\alp) \op{deg}(\be)} f^*(\be) \bullet \alp$ in the situation above).\\

$\bB_*(X):= \bB(X \to pt)$ becomes a covariant functor for {\it confined}  morphisms and 
$\bB^*(X) := \bB(X  \xrightarrow{id}  X)$ becomes a contravariant ring valued functor for {\it any} morphisms, with $\bB_*(X)$ a left $\bB^*(X)$-module under the product
$\cap:=\bullet: \bB^*(X)\otimes \bB_*(X)\to \bB_*(X)$.
As to a possible grading, one sets  $\bB_i(X):= \bB^{-i}(X  \to  {pt})$ and $\bB^j(X):= \bB^j(X  \xrightarrow{id}  X)$ so that $\bB^*(X)$ becomes a graded ring with
$\cap: \bB^j(X)\otimes \bB_i(X)\to \bB_{i-j}(X)$.\\

The following notion of an \emph{orientation} makes $\bB_*$ a contravariant functor and $\bB^*$ a covariant functor with the corresponding Gysin (or transfer) homomorphisms:

\begin{defn}\label{canonical}(\cite[ Part I, Definition 2.6.2]{Fulton-MacPherson}) Let $\Cal S$ be a class of maps in $\Cal V$, which is closed under compositions and contains all identity maps. Suppose that to each $f: X \to Y$ in $\Cal S$ there is assigned an element
$\theta(f) \in \bB(X  \xrightarrow {f} Y)$ satisfying that
\begin{enumerate}
\item [(i)] $\theta (g \circ f) = \theta(f) \bullet \theta(g)$ for all $f:X \to Y$, $g: Y \to Z \in \Cal S$ and

\item [(ii)] $\theta(\op {id}_X) = 1_X $ for all $X$ with $1_X \in \bB^*(X):= 
\bB(X  \xrightarrow{\op {id}_X} X)$ the unit element.
\end{enumerate}
Then $\theta(f)$ is called  an {\it orientation} of $f$.  If we need to refer to which bivariant theory we consider, we denote $\theta_{\bB}(f)$ instead of the simple notation $\theta (f)$. 
\end{defn} 
\begin{rem} Since there can be different choices of such orientations (
e.g., compare with our ``twisting'' construction later on),
we prefer to call the above $\theta$ simply an \emph{orientation} for what is called 
a ``canonical orientation'' in \cite{Fulton-MacPherson}. If we want to emphasize the class $\Cal S$, it is called an \emph{$\Cal S$-orientation}, 
and if we want to emphasize the bivariant theory $\bB$ as well, it is called a \emph{$\bB$-valued $\Cal S$-orientation}.
\end{rem}

For example the class $\Cal S$ of {\em smooth} morphisms in the algebraic or analytic geometric context has 
orientations for all the bivariant theories mentioned in the introduction,
with all cartesian squares independent.

\begin{pro} For the composite $X  \xrightarrow{f}  Y \xrightarrow{g}  Z$,   if $f \in \m S$ has 
an orientation $\theta_{\bB}(f)$, then we have the Gysin homomorphism (or transfer) defined by $f^!(\alp) :=\theta(f) \bullet \alp$:
$$f^!: \bB(Y  \xrightarrow{g}  Z) \to \bB(X \xrightarrow{gf}  Z),$$
which is functorial, i.e., $(gf)^! =f^! g^!$ and $id^!=id$. In particular, when $Z = pt$, we have the Gysin homomorphism:
$$f^!: \bB_*(Y) \to \bB_*(X).$$
\end{pro}

\begin{pro} For an independent square 
$$\CD
X' @> g' >> X \\
@V f' VV @VV {f}V\\
Y'@>> g > Y, \endCD
$$
if $g \in \m C \cap \m S$ and $g$ has 
an orientation $\theta_{\bB}(g)$, then we have the Gysin homomorphism defined by $g_!(\alp) :=g'_*(\alp \bullet \theta(g))$:
$$g_!: \bB(X'  \xrightarrow{f'}  Y') \to \bB(X \xrightarrow{f}  Y),$$
which is functorial, i.e., $(gf)_! = g_! f_!$ and $id_!=id$. In particular, for an independent square 
$$\CD
X @> f >> Y \\
@V {\op {id}_X} VV @VV {\op {id}_Y}V\\
X @>> f > Y, \endCD
$$
with $f \in \m C \cap \m S$, we have the Gysin homomorphism:
$$f_!: \bB^*(X) \to \bB^*(Y).$$ 
\end{pro}

The symbols $f^!$ and $g_!$ should carry the information of $\Cal S$ and the 
orientation $\theta$, but it will be usually omitted if it is not necessary to be mentioned. \\

Suppose that we have  a Grothendieck transformation $\ga: \bB \to \bB'$ of two bivariant theories  $\bB, \bB'$.
This induces natural transformations $\ga_*: \bB_* \to \bB_*'$ and $\ga^*: \bB^* \to {\bB'}^*$, i.e., we have the following commutative diagrams:\\
For any morphism $f: X \to Y$ we have the commutative diagram
$$\CD
\bB^*(X) @> \ga^* >> \bB'^*(X) \\
@V {f^*} VV @VV {f^*}V\\
\bB^*(Y) @>> \ga^*> \bB'^*(Y) . \endCD
$$
For a confined morphism $f:X \to Y$ we have the  commutative diagram
$$\CD
\bB_*(X) @> \ga_* >> \bB'_*(X) \\
@V {f_*} VV @VV {f_*}V\\
\bB_*(Y) @>> \ga_*> \bB'_*(Y) . \endCD 
$$ 
And these are related by the \emph{module property} 
$$\ga_*(\beta\cap \alpha) = \ga^*(\beta) \cap \ga_*(\alpha) \quad \text{for all}
\quad \beta\in \bB^*(X), \alpha \in \bB_*(X).$$

Assume now that $f: X\to Y$ has 
an orientation for both bivariant theories.
Then a bivariant element $u_f \in \bB'^*(X) = \bB'(X \xrightarrow {\op {id}_X} X)$ with
$$\ga (\theta_{\bB}(f)) = u_f \bullet \theta_{\bB'}(f)$$
 is called a \emph{Riemann--Roch formula} (see \cite{Fulton-MacPherson}) 
 comparing these 
 orientations with respect to the bivariant theories $\bB, \bB'$. 
Such a Riemann--Roch formula  gives rise to the following (wrong-way) commutative diagrams with respect to the above two Gysin homomorphisms $f_!, f^!$ :
$$ 
\begin{CD}
\bB^*(X) @> \ga^* >> \bB'^*(X) \\
@V {f_!} VV @VV {f_!( \;-\; \bullet u_f)}V\\
\bB^*(Y) @>> \ga^*> \bB'^*(Y) . 
\end{CD}
\hspace{2cm} \begin{CD}
\bB_*(Y) @> \ga_* >> \bB'_*(Y) \\
@V {f^!} VV @VV {u_f \bullet f^!}V\\
\bB_*(X) @>> \ga_*> \bB'_*(X) . 
\end{CD}$$

The most important and motivating example of such a Grothendieck transformation is Baum--Fulton--MacPherson's bivariant \emph{Riemann--Roch transformation} (\cite [Part II]{Fulton-MacPherson}):
$$\tau: \bK_{alg} \to \bH\otimes \bQ\:,$$ 
or its algebraic counterpart of \cite[Example 18.3.19]{Fulton-book}. Here $\m V=\m V^{qp}_{k}$
is the category of quasi-projective varieties over a base field $k$ of any characteristic,
with  $\bH=CH$ the bivariant operational Chow groups, or $\bH$
the even degree bivariant homology in case $k=\bC$. The independent squares in this context are the \emph{Tor-independent} fiber squares.
 $\bK_{alg}$ is the bivariant algebraic K-theory of relative perfect complexes, 
so that ${\bK_{alg}}_*(X) = K_0(X)$ is the Grothendieck group of coherent sheaves and ${\bK_{alg}}^*(X) = K^0(X)$ is the Grothendieck group of algebraic vector bundles. 
The associated contravariant transformation is the 
\emph{Chern character} 
$$\tau^*=ch: K^0(X) \to H^*(X)\otimes \bQ,$$
and the associated covariant transformation is the
\emph{Todd class transformation}
 $$\tau_*=td_*: G_0(X)\to H_*(X)\otimes \bQ,$$
which is functorial for proper morphisms $f:X \to Y$. Moreover, they
are related by the \emph{module property} 
 \begin{equation}\label{module}
 td_*(\beta\cap \alpha) = ch^*(\beta) \cap td_*(\alpha) \quad \text{for all}
\quad \beta\in K^0(X), \alpha \in G_0(X).
\end{equation}

This generalizes the original Grothendieck--Riemann--Roch Theorem and Hirzebruch--Riemann--Roch Theorem.
Both bivariant theories $\bK_{alg}$ and $H_*(-)\otimes \bQ$ are 
oriented for the
class  $\Cal S$ of smooth (or more generally of local complete intersection) morphism,
with $\theta_{\bK}(f)= \m O_f:=[\m O_X]\in \bK_{alg}(X \xrightarrow {f} Y)$ the class of the structure sheaf, and $\theta_{\bH}(f)=[f]\in \bH(X \xrightarrow {f} Y)$ the corresponding
``relative fundamental class''. And these are
related by the \emph{Riemann--Roch formula} 
\begin{equation}\label{RR-formula} \tau(\m O_f) = td(T_f)\bullet [f] \:,
\end{equation}
with $u_f:=td(T_f)\in H^*(X)\otimes \bQ$ (compare with \cite [(*) on p.124]{Fulton-MacPherson}
for  $\bH$ the bivariant homology in case $k=\bC$. For $\bH=CH$ the bivariant Chow group and
$k$ of any characteristic, this follows from \cite[Theorem 18.2]{Fulton-book} as we explain in the last section of our paper).
Here $T_f$ is the (virtual) tangent bundle of $f$. This implies the following two results:\\\\
\underline {SGA 6-Riemann--Roch Theorem}: The following diagram commutes
for a proper smooth morphism $f: X \to Y$:
\begin{equation}\CD
K(X) @> ch >> H^*(X)\otimes \bQ \\
@V {f_!} VV @VV {f_!(td(T_f) \cup \;-\;) }V\\
K(Y) @>> ch> H^*(Y)\otimes \bQ . \endCD
\end{equation}
\underline {Verdier--Riemann--Roch Theorem}: 
 The following diagram commutes
for a smooth morphism $f: X \to Y$: 
\begin{equation}\CD
G_0(Y) @> td_* >> H_*(Y)\otimes \bQ \\
@V {f^!} VV @VV {td(T_f)\cap f^!}V\\
G_0(X) @>> td_* > H_*(X)\otimes \bQ  . \endCD
\end{equation}

Of course both formulae are more generally true for $f$ a local complete intersection morphism,
which is special to the Grothendieck transformation $\tau$. In this paper
 only the case of a smooth morphism will be used, and then similar results are also true
 for the other considered Grothendieck transformations.
It should also be remarked that \emph{one motivation of Fulton--MacPherson's bivariant theory was to unify the above three Riemann--Roch theorems ... } (see \cite[Part II, \S 0.1.4]{Fulton-MacPherson}).

\begin{defn}\label{stable} (i) Let $\Cal S$ be another class of maps in $\Cal V$ , called ``specialized maps" (e.g., smooth maps in algebraic geometry), 
which is closed under composition and under base change and containing all identity maps. Let $\bB$ be a bivariant theory. If $\Cal S$ has 
orientations in $\bB$, then we say that $\Cal S$ is {\it 
$\bB$-oriented }and an element of $\Cal S$ is called a {\it 
$\bB$-oriented} morphism.

(ii) Assume furthermore, that the orientation $\theta$ on $\Cal S$ satisfies for any independent square with $f \in \Cal S$
$$
\CD
X' @> g' >> X\\
@Vf'VV   @VV f V \\
Y' @>> g > Y
\endCD
$$
the  condition
\begin{equation}\label{nice}
\theta (f') = g^* \theta (f)
\end{equation}
(which means that the orientation $\theta$ is preserved under the pullback operation).
Then we call $\theta$ a {\it 
stable orientation} and say that $\Cal S$ is {\it 
stably $\bB$-oriented}. Similarly an element of $\Cal S$ is called a {\it 
stably $\bB$-oriented} morphism.
\end{defn}

Consider for example the class $\Cal S$ of all {\it smooth} morphisms for $\m V=\m V^{(qp)}_{k}$
the category of (quasi-projective) varieties over a base field $k$ of any characteristic,
with all fiber squares as the independent squares. Then this class has a 
stable orientation $\theta$ with respect to $\bK_{alg}$ or $CH$ in any characteristic
(with $\theta(f)=\m O_f$ or $[f]$),
to $\tilde{\bF}$ in characteristic zero (with $\theta(f)=\jeden_f$) and to
$\bF$ or bivariant homology $\bH$ for $k=\bC$ (with $\theta(f)=\jeden_f$ or $[f]$).

\section{A universal bivariant theory on the category of  varieties}

Let $\Cal V$ be the category $\m V=\m V^{(qp)}_{k}$
of (quasi-projective) varieties over a base field $k$ of any characteristic,
or the category $\m V=\m V^{an}_{c}$ of compact reduced analytic spaces,
with all fiber squares as the independent squares. 
As the ``confined" resp. ``specialized" maps we take the class $\Cal Prop $ of {\it proper}
resp. $\Cal Sm$ of {\it smooth} morphisms.

\begin{thm}\label{UBT} We define 
$$\bM(\m V/X  \xrightarrow{f}  Y)$$
to be the free abelian group generated by the set of isomorphism classes of proper morphisms $h: W \to X$  such that the composite of  $h$ and $f$ is a smooth morphism:
$$h \in \Cal Prop  \quad \text {and} \quad f \circ h: W \to Y \in \Cal Sm.$$
Then the association $\bM$ is a \emph{bivariant theory}, if the three operations are defined as follows:

{\bf Product operation}: For morphisms $f: X \to Y$ and $g: Y
\to Z$, the product operation
$$\bullet: \bM (\m V/X  \xrightarrow{f}  Y) \otimes \bM (\m V/ Y  \xrightarrow{g}  Z) \to
\bM ( \m V/X  \xrightarrow{gf}  Z)$$
is  defined for $[V \xrightarrow{p}  X]  \in \bM (\m V/X  \xrightarrow{f}  Y)$ and $[W  \xrightarrow{k}  Y]  \in \bM (\m V/ Y  \xrightarrow{g}  Z)$ by 
$$[V \xrightarrow{p}  X] \bullet [W  \xrightarrow{k}  Y]  :=  [V'  \xrightarrow{p \circ {k}''}  X], $$
and bilinearly extended. Here we consider the following fiber squares
\begin{equation}\label{cd:product}\CD
V' @> {p'} >> X' @> {f'} >> W \\
@V {k''}VV @V {k'}VV @V {k}VV\\
V@>> {p} > X @>> {f} > Y @>> {g} > Z .\endCD
\end{equation}

{\bf Pushforward operation}: For morphisms $f: X \to Y$
and $g: Y \to Z$ with $f \in \m Prop$, the pushforward operation
$$f_*: \bM (\m V/ X  \xrightarrow{gf} Z) \to \bM (\m V/ Y  \xrightarrow{g}  Z) $$
is  defined by
$$f_*([V \xrightarrow{p}  X]) := [V  \xrightarrow{f \circ p}  Y]$$
and linearly extended. 

{\bf Pullback operation}: For an independent square
$$\CD
X' @> g' >> X \\
@V f' VV @VV f V\\
Y' @>> g > Y, \endCD
$$
the pullback operation
$$g^* : \bM (\m V/ X  \xrightarrow{f} Y) \to \bM(\m V/ X'  \xrightarrow{f'} Y') $$
is  defined by
$$g^*([V  \xrightarrow{p}  X] ):=  [V'  \xrightarrow{p'}  X']$$
and linearly extended. Here we consider the following fiber squares:
\begin{equation}\label{cd:pullback}\CD
V' @> g'' >> V \\
@V {p'} VV @VV {p}V\\
X' @> g' >> X \\
@V f' VV @VV f V\\
Y' @>> g > Y. \endCD
\end{equation}
\end{thm}
The proof is left for the reader. Note that $\theta(f):=[X \xrightarrow {\op {id}_X} X]$
for the smooth morphism $ f: X\to Y$defines a 
stable orientation on $\bM(\m V/-)$.
We call the bivariant theory $\bM(\m V/-)$ a 
\emph{pre-motivic bivariant relative Grothendieck group} on the category $\m V$ of varieties.

\begin{rem} 
(1) ${\bM}_*(\m V/X) = \bM(\m V/X \to pt)$ is the free abelian group generated by the isomorphism classes $[V \xrightarrow{h}  X]$, where $h$ is proper and $V$ is smooth.  ${\bM}_*(\m V/-)$ is a covariant functor for proper morphisms, i.e., if $f: X \to Y$ is proper, we have the covariant pushforward $$f_*: {\bM}_*(\m V/X) \to {\bM}_*(\m V/Y)\:.$$
 ${\bM}_*(\m V/-)$ is also a contravariant functor for  smooth morphisms, i.e., if $f:X \to Y$ is a smooth morphism, we have the contravaraint Gysin homomorphism 
$$f^!: {\bM}_*(\m V/Y) \to {\bM}_*(\m V/X) \:.$$
(2) ${\bM}^*(\m V/X) = \bM(\m V/X \xrightarrow{\op {id}_X}  X)$ is the free abelian group generated by the isomorphism classes $[V \xrightarrow{h}  X]$, 
where $h$ is \underline {proper and smooth}. 
It gets a ring structure $\cup$ by fiber products, with unit 
$1_X=[X \xrightarrow {\op {id}_X} X]$.
 ${\bM}^*(\m V/-)$ is a contravariant functor for any morphism, i.e., for any morphism $f:X \to Y$ we have the contravariant pullback (preserving $\cup$ and the units)
 $$f^*: {\bM}^*(\m V/Y) \to {\bM}^*(\m V/X)\:.$$ 
 ${\bM}^*(\m V/-)$ is also a covariant functor for  morphisms which are smooth and proper, i.e., if $f:X \to Y$ is a smooth proper morphism, we have the covariant Gysin homomorphism $$f_!: {\bM}^*(\m V/X) \to {\bM}^*(\m V/Y)\:.$$
(3) The bivariant product induces the following ``cap product":
$$\cap :{\bM}^*(\m V/X) \times {\bM}_*(\m V/X) \to {\bM}_*(\m V/X).$$
In particular, when $X$ itself is  a \emph{smooth} variety, with $[X]:=\cap[X \xrightarrow{\op {id}_X}  X]\in {\bM}_*(\m V/X)$, we have the ``Poincar\'e duality" homomorphism
$$\cap [X]: {\bM}^*(\m V/X) \to {\bM}_*(\m V/X)\:,$$
which is nothing but $[W \xrightarrow{k}  X ]\cap[X] = [W \xrightarrow{k}  X ].$
More  generally, the isomorphism class $[V \xrightarrow{h}  X]\in {\bM}_*(\m V/X)$ of any proper morphism $h:V \to X$ from a \emph{smooth} variety $V$ to $X$ gives rise to the  homomorphism
$$\cap{[V \xrightarrow{h}  X]}: {\bM}^*(\m V/X) \to {\bM}_*(\m V/X)$$
defined by $[W \xrightarrow{k}  X ]\cap[V \xrightarrow{h}  X]= [W \times _X V \to X].$
\end{rem}

The bivariant theory $\bM(\m V/-)$ has the following universal property (see \cite[Theorem 3.1]{Yokura-obt} for the proof of a more general result):

\begin{thm}\label{univ}  Let $\bB$ be a  bivariant theory on  $\m V$ such that a smooth morphism $f: X\to Y$ has a 
stable orientation $\theta(f)\in \bB(f)$.
Then there exists a unique Grothendieck transformation
$$\ga:=\ga_{\theta}: \bM(\m V/-) \to \bB(-)$$
satisfying the normalization condition that for a smooth morphism $f:X \to Y$ the following identity holds in $\bB(X \xrightarrow {f} Y)$:
$$\ga([X \xrightarrow {\op {id}_X} X]) =  \theta(f).$$
\end{thm}

Note that in \cite{Yokura-obt} only {\em commutative} bivariant theories are considered, but the result and proof of \cite[Theorem 3.1]{Yokura-obt}
works without this assumption.

\begin{cor}\label{twisting}
Let $c\ell: Vect(-) \to  \bB^*(-)$ be a contravariant functorial characteristic  class of algebraic (or analytic) vector bundles
with values in the associated cohomology theory, which is multiplicative in the sense that $c\ell(V) = c\ell(V') c\ell(V'')$ for 
any short exact sequence of vector bundles $0\to V'\to V \to V'' \to 0$, with $c\ell(T_{pt})=1_{pt}\in \bB^*(\{pt\})$. 
Assume that $c\ell$ commutes with the stable orientation $\theta$, i.e.
$$\theta(f)\bullet cl(V)=f^*cl(V)\bullet \theta(f)$$ 
for all smooth morphism $f: X\to Y$ and $V\in Vect(Y)$.
Then there exists a unique Grothendieck transformation
$$\ga_{c\ell}: \bM(\m V/-) \to \bB(-)$$
satisfying the normalization condition that for a smooth morphism $f:X \to Y$ the following identity holds in $\bB(X \xrightarrow {f} Y)$:
$$\ga_{c\ell}([X \xrightarrow {\op {id}_X} X]) = c\ell(T_f) \bullet \theta(f).$$
Here $T_f$ is the relative tangent bundle  of the smooth morphism $f$.
\end{cor}

This follows from Theorem \ref{univ} by using the next result (similar twisting constructions are due to Quillen \cite{Quillen} 
(resp., Levine-Morel \cite[\S 4.1.9]{LM}) in the context of
complex oriented (co)homology theories (resp., oriented Borel-Moore (weak) homology theories):

\begin{lem}
 The definition  $\theta'(f):=c\ell(T_f) \bullet \theta(f)$ for a smooth morphism 
$f: X\to Y$ defines a new ``twisted" stable orientation.
\end{lem}

\begin{proof} First note that $T_{id_X}$ is the zero vector bundle $p^*T_{pt}$ for $p: X\to pt$ the constant map so that by functoriality
$c\ell(T_{id})=  p^*c\ell(T_{pt}) = p^*(1_{pt})=1_X \in \bB^*(X)$. This implies (ii): $\theta'(id_X)=1_X\in \bB^*(X)$ for all $X$.

Let us now proof the multiplicativity (i):
$$\theta'(g\circ f) = \theta'(g) \bullet  \theta'(f)$$
for all smooth morphism $f: X\to Y$ and $g: Y\to Z$.
Here we have $c\ell (T_{gf})= c\ell (T_{f}) \bullet f^* c\ell(T_{g})$ by the functoriality and multiplicativity of $c\ell$, due to the short exact sequence of vector bundles
$$0\to T_{f } \to T_{gf} \to f^*T_{g}\to 0\:.$$
Similarly $\theta(gf)=\theta(f) \bullet \theta(g) $, since $\theta$ is a canonical orientation. Moreover, $c\ell$ commutes by assumption with the 
orientation $\theta$
so that
$$\theta(f) \bullet c\ell(T_{g}) = f^* c\ell(T_{g}) \bullet \theta(f).$$

So we get
\begin {align*}
\theta'(g\circ f) &:= c\ell(T_{g\circ f}) \bullet \theta(g\circ f)\\
& = \left( c\ell (T_{f}) \bullet f^* c\ell(T_{g})\right) \bullet \left(  \theta(f) \bullet \theta(g) \right) \\
& =  c\ell (T_{f}) \bullet \left( f^* c\ell(T_{g}) \bullet   \theta(f) \right) \bullet \theta(g)  \\
& =  \left( c\ell (T_{f}) \bullet \theta(f) \right) \bullet \left( c\ell(T_{g}) \bullet \theta(g) \right) \\
& = \theta'(g) \bullet  \theta'(f) \:.
\end{align*}

Finally we show that $\theta'$ is a stable orientation, i.e. (iii):
$$\theta'(f')=g^* \theta'(f)$$
in the context of Definition \ref{stable}(ii). This follow from $T_{f'}\simeq g^*(T_f)$ by the functoriality 
$ c\ell(T_{f'}) = g^*c\ell(T_{f})$ of $ c\ell$ and the stability $\theta(f')=g^* \theta(f)$ of $\theta$:

\begin{align*}
\theta'(f')&:= c\ell(T_f') \bullet \theta(f')\\
&= g^*c\ell(T_{f}) \bullet g^* \theta(f)\\
&= g^*\left( c\ell(T_{f}) \bullet  \theta(f) \right)\\
&= g^* \theta'(f)\:.
\end{align*}
\end{proof}

Note that the assumption, that the characterisic class  $c\ell$ commutes with the 
orientation $\theta$, is true for $\bB$ commutative, or $\bB$ graded-commutative with
$c\ell$ taking values in even degree cohomology classes.
Similarly it is true for the trivial class $c\ell(V)=1$ the unit in $\bB^*(-)$,
as well as for $\bB=CH$ the bivariant Chow homology , with $c\ell$ a ``usual"
multiplicative characteristic class given in terms of Chern class operators as in
\cite[\S 3.2]{Fulton-book}. This covers all cases we need in this paper. Finally, the Grothendieck transformation 
$$\ga_{c\ell}: \bM(\m V/-) \to \bB(-)$$
from Corollary \ref{twisting}
satisfies by the normalization condition 
$$\ga_{c\ell}([X \xrightarrow {\op {id}_X} X]) = c\ell(T_f) \bullet \theta(f)$$
the \emph{Riemann-Roch formula} with $u_f=c\ell(T_f)$  for a smooth morphism $f:X \to Y$.
So by the general theory we get the\\ 

\underline {SGA 6 -type Riemann--Roch Theorem}: The following diagram commutes for a 
proper smooth morphism $f: X \to Y$:
$$\CD
{\bM}^*(\m V/X) @>  {\ga_{c\ell}}^* >> \bB^*(X)\\
@V {f_!} VV @VV {f_!(c\ell(T_f) \cup \;-\;) }V\\
{\bM}^*(\m V/Y) @>> {\ga_{c\ell}}^* > \bB^*(Y) . \endCD
$$

\underline {Verdier-type Riemann--Roch Theorem}:  The following diagram commutes for a smooth morphism $f: X \to Y$:
$$\CD
{\bM}_*(\m V/X) @> {\ga_{c\ell}}_* >> \bB_*(X)\\
@V {f^!} VV @VV {c\ell(T_f)\cap f^!}V\\
{\bM}_*(\m V/Y) @>>  {\ga_{c\ell}}_* > \bB_*(Y) . \endCD
$$\\

\begin{rem}
\begin{enumerate}
\item  $\ga_{c\ell}:\bM(\m V/X  \xrightarrow{f}  Y) \to \bB(X  \xrightarrow{f}  Y)$ can be called a \emph {bivariant pre-motivic characteristic class transformation}. When $Y$ is a point $pt$, 
$${\ga_{c\ell}}_*:  \bM (\m V/ X \to pt) \to \bH(X \to pt) = \bB_*(X)$$
 is the \emph{unique natural transformation} satisfying the \emph{normalization condition} that for a smooth variety
$${\ga_{c\ell}}_*([X \xrightarrow {id_X} X]) = c\ell(TX) \cap [X].$$
In other words, this gives rise to a \emph{pre-motivic characteristic class transformation for singular varieties}. 
In a sense, this could be also a very general answer to the forementioned MacPherson's question about the existence of a unified theory of 
characteristic classes for singular varieties. 

As mentioned in Remark \ref{univ-BM},  ${\bM}_ *(\m V/-)$ is in fact  a {\em universal Borel Moore functor} 
with product (but without an additivity property), see \cite{Yokura-obt}.
\\
\item In particular, we have the following commutative diagrams:
$$\xymatrix{
&{\bM}_* (\m V/ X) \ar [dl]_{\epsilon} \ar [dr]^{{\ga_{c}}_*} \\
{F(X) } \ar [rr] _{c_*}& &  H_*(X)},$$
with $H_*(X)=CH_*(X)$ in the algebraic context over a base field of characteristic zero,
or $H_*(X)=H^{BM}_{2*}(X)$ in the complex algebraic or compact complex analytic context.
Here $\epsilon ([V \xrightarrow{h}  X]) := h_* \jeden_V$.
$$\xymatrix{
&  {\bM}_* (\m V/ X) \ar [dl]_{mC_0} \ar [dr]^{{\ga_{td}}_*} \\
{G_0(X) } \ar [rr] _{td_*}& &  H_*(X)\otimes \bQ},$$
with $H_*(X)=CH_*(X)$ in the algebraic context over a base field of any characteristic,
or $H_*(X)=H^{BM}_{2*}(X)$ in the complex algebraic or compact complex analytic context.
Here $mC_0([V \xrightarrow{h}  X]) := [h_* \mathcal O_V]$.
$$\xymatrix{
& {\bM}_* (\m V/ X) \ar [dl]_{sd} \ar [dr]^{{\ga_L}_*} \\
{\Omega_{sd}(X) } \ar [rr] _{L_*}& &  H_{2*}^{BM}(X)\otimes \bQ.}
$$
Here $X$ has to be a compact complex algebraic or analytic variety,
with $$sd([V \xrightarrow{h}  X]) := [h_* \bQ_V[dim(X)]]\:.$$

Note that all these covariant theories come from a suitable bivariant theory, except $\Omega_{sd}(X)$. So the right slant arrows follow e.g. from
Corollary \ref{twisting} applied to the canonical orientation $\theta(f)=[f]\in \bH(f)$ given by the relative fundamental class of the smooth morphism $f$.
As mentioned already before, the characteristic classes $c\ell= c^*, td^*$ and $L^*$ are multiplicative and commute with $\theta$ by general reasons.
The first two left slant arrows follow from Theorem \ref{univ} applied to the following canonical orientation of a smooth morphism $f$: 
$\theta(f)=\jeden_f\in \bF$ resp. 
$\jeden_f\in \tilde{\bF}$, and $\theta(f)= \m O_f \in \bK_{alg}(f)$. The third left slant arrows $sd$ follows e.g. from the universal property 
of ${\bM}_* (\m V/ -)$ as a universal  Borel Moore functor (or by direct construction). 
\item  It follows from Hironaka's resolution of singularities (\cite{Hironaka}) that there exists a surjection
$${\bM}_* (\m V/ X) \to K_0(\m V /X)$$
in the algebraic context over a base field of characteristic zero,
or in the  compact complex analytic context.
As already explained in the introduction, it then turns out that if (under a certain requirement) the natural transformation  ${\ga_{c\ell}}_*:  {\bM}_* (\m V/ X)\to  H_*(X)\otimes \bQ[y]$ can be pushed down to the relative Grothendieck group $K_0(\m V /X)$, then it has to be the Hirzebruch class transformation, i.e., the following diagram commutes:
$$\xymatrix{
& {\bM}_* (\m V/ X) \ar [dl]_{q} \ar [dr]^{{\ga_{c\ell}}_*} \\
{K_0(\Cal V/X) } \ar [rr] _{{T_y}_*}& &  H_*(X)\otimes \bQ[y].}
$$
And one of the main results of our previous paper \cite{BSY1} claims that in this context the above three diagrams also commute with ${\bM}_* (\m V/ X)$ being replaced by the smaller group $K_0(\m V/X)$.
\end{enumerate}
\end{rem}

Thus we are led to the following natural problem:

\begin{prob}\label{problem} Formulate a reasonable bivariant analogue $\bK_0(\m V/ X\xrightarrow{f} Y)$ of the relative Grothendieck group $K_0(\m V/ X)$ so that the following hold:
\begin{enumerate}
\item There is a natural group homomorphism $q: \bK_0(\m V/ X\xrightarrow{} pt) \to K_0(\m V/X)$, which is an isomorphism in the algebraic context over a base field of characteristic zero, or in the  compact complex analytic context.
\item $\bB q: \bM(\m V/ X\xrightarrow{f} Y)  \to \bK_0(\m V/ X\xrightarrow{f} Y)$ is a certain quotient map, which specializes for $Y$  a point to the quotient map $q: {\bM}_*(\m V/ X)  \to K_0(\m V/ X)$.
\item $T_y:\bK_0(\m V/ X\xrightarrow{f} Y) \to \bH(X\xrightarrow{f} Y)\otimes \bQ[y]$ is a Grothendieck transformation, which specializes for $Y$ a point (in the algebraic context over a base field of characteristic zero, or in the  compact complex analytic context) to the motivic Hirzebruch class transformation
${T_y}_*:K_0(\Cal V/X) \to H_*(X)\otimes \bQ[y]$.
\item The following diagram commutes:
$$\xymatrix{
& \bM(\m V/ X\xrightarrow{f} Y) \ar [dl]_{\bB q} \ar [dr]^{\ga_{T_{y}^*}} \\
{\bK_0(\m V/ X\xrightarrow{f} Y)} \ar [rr] _{T_y}& &  \bH(X\xrightarrow{f} Y)\otimes \bQ[y].}
$$
\end{enumerate}
\end{prob}

If such a bivariant theory $\bK_0(\m V/ X\xrightarrow{f} Y)$ is obtained, then its associated contravariant functor $K^0(\m V/ X): = \bK_0(\m V/ X\xrightarrow{\op {id}_X} X)$ can be considered as a contravariant counterpart of the relative Grothendieck group $K_0(\m V/X)$
(at least in the algebraic context over a base field of characteristic zero, or in the  compact complex analytic context). Similarly, the natural transformation $T_y^*: K^0(\m V/ -) 
\to H^*(-)\otimes \bQ[y]$ is a contravariant counterpart of the Hirzebruch class transformations
$T_{y*}$ satisfying the \emph{module property}.

\section{A bivariant relative Grothendieck group $\bK_0(\m V/ X\xrightarrow{f} Y)$}

First we recall the following result of Franziska Bittner \cite{Bittner}:
\begin{thm}[Bittner]\label{bittner} Let $K_0(\m V/X)$ be the relative Grothendieck group of varieties over
$X\in obj(\m V)$, with $\m V=\m V^{(qp)}_k$ (resp. $\m V=\m V^{an}_c$) the category of
(quasi-projective) algebraic (resp. compact complex analytic) varieties over a base field $k$
of characteristic zero.

Then $K_{0}(\m V/X)$ is isomorphic to ${\bM}_*(X)$ modulo the
``blow-up'' relation
\begin{equation} \label{eq:bl}
[\emptyset \to X]=0 \quad \text{and} \quad
[Bl_{Y}X'\to X] - [E\to X]= [X'\to X] - [Y\to X] \:,\tag{bl}
\end{equation}
for any cartesian diagram (which shall be called the ``blow-up diagram" from here on)
\begin{displaymath} \begin{CD}
E @> i'>> Bl_{Y}X' \\
@VV q' V  @VV q V \\
Y @> i >> X' @> f >> X \:,
\end{CD} \end{displaymath}
with $i$ a closed embedding of smooth spaces and 
$f:X'\to X$ proper. Here $Bl_{Y}X'\to X'$ is the blow-up of $X'$ along 
$Y$ with exceptional divisor $E$. Note that all these spaces other than $X$ are
also smooth (and quasi-projective in case $X', Y\in ob(\m V^{qp}_k)$).
\end{thm}
The proof of this theorem requires Abramovich et al's ``Weak Factorisation Theorem" 
\cite{AKMW, W}. 
The kernel of the canonical quotient map $q:{\bM}_* (\m V/ X) \to  K_0(\m V/X)$ is the subgroup $BL(\m V/X)$ of ${\bM}_* (\m V/ X)$ generated by 
$$[Bl_{Y}X'\to X] - [E\to X] -[X'\to X] + [Y\to X]$$
for any blow-up diagram as above.

Thus what we want is a bivariant analogue of the subgroup $BL(\m V/X)$. For that purpose we first observe the following result, working in the category 
$\m V=\m V^{(qp)}_k$ (resp. $\m V=\m V^{an}_c$) of
(quasi-projective) algebraic (resp. compact complex analytic) varieties over a base field $k$
of any characteristic.
 
\begin{lem} \label{key-lemma} Let  $h:X' \to X$ be a \emph{smooth} morphism, with $i:S \to X'$ a closed embedding such that the composite $h \circ i : S \to X$ is also \emph{smooth} morphism. 
Consider the cartesian diagram
\begin{equation}\label{relative bl}
\begin{CD}
E @> i'>> Bl_{S}X' \\
@V q' VV  @VV q V \\
S @>> i > X' @>>h > X \:,
\end{CD}
\end{equation}
with $q: Bl_{S}X' \to X'$ the blow-up of $X'$ along $S$ and $q':E \to S$ the exceptional divisor map. Then: 
\begin{enumerate}
\item  $h \circ q: Bl_{S}X' \to X$  and $h\circ q\circ i': E \to X$ are also \emph{smooth} morphisms,
with $ Bl_{S}X', E$ quasi-projective in case $X', S\in ob(\m V^{qp}_k)$.
\item This blow-up diagram \emph{commutes with any base change} in $X$, i.e. the 
corresponding fiber-square induced by pullback along a morphism $\tilde{X}\to X$
is isomorphic to the corresponding  blow-up diagram of $\tilde{S}\to \tilde{X}'$.
\item The closed embeddings $i,i'$ are \emph{regular} embeddings, and the projection map $q$ as well as $i,i'$ are of \emph{finite Tor-dimension}.
\end{enumerate}
\end{lem}

\begin{proof}
Note that all results are (\'etale) local in $X'$.
Since both morphisms $h: X' \to X$ and $S \to X' \to X$ are smooth, we can assume that
$h$ is the projection $h=pr_2: X'= \bA^n \times X \to X$, with $i: S= \bA^m \times X \to
\bA^n \times X $ induced from a standard inclusion $\bA^m \hookrightarrow \bA^n$ of affine spaces
($m\leq n$), and the blow-up diagram (\ref{relative bl}) isomorphic to
$$\begin{CD}
E \times X @> i'>> Bl_{\bA^m}\bA^n \times X \\
@V q' VV  @VV q V \\
\bA^m \times X @>> i > \bA^n \times X @>>h=pr_2 > X \:.
\end{CD}$$
Here we use the fact that 
$$Bl_{  \bA^m \times X}(\bA^n \times X)\simeq Bl_{\bA^m}\bA^n \times X \:,$$
since blowing up commutes with flat base change for the flat projection map 
$h=pr_2: X'=\bA^n \times X \to X$. Then (1) and (3) are well known, whereas (2)
follows again from the fact that blowing up commutes with flat base change for the flat projection maps 
$h=pr_2: X'=\bA^n \times X \to X$ and
$\tilde{h}=pr_2: \tilde{X}'=\bA^n \times \tilde{X} \to \tilde{X}$.
\end{proof}

Now we are ready to define a bivariant analogue ${\mathbb {BL}(\m V/X \xrightarrow{f}  Y)}$ of the subgroup $BL(\m V/X)$ and thus a bivariant analogue $\bK_0(\m V/X \xrightarrow{f}  Y)$ of $K_0(\m V/X)$.

\begin{defn} For a morphism $f: X \to Y$ in the category $\m V=\m V^{(qp)}_k$ or 
$\m V=\m V^{an}_c$, we consider a blow-up diagram
$$\begin{CD}
E @> i'>> Bl_{S}X' \\
@VV q' V  @VV q V \\
S @> i >> X' @> h >> X  @> f >> Y\:,
\end{CD} $$
with $h$ proper and $i$ a closed embedding such that
$f \circ h$ as well as $f \circ h \circ i$ are smooth. 
 
Let $\mathbb {BL}(\m V/ X \xrightarrow{f}  Y)$ be the free abelian subgroup of 
$\bM(\m V/X \xrightarrow{f}  Y)$ generated by 
\begin{equation} \label{eq:rbl}
 [Bl_{S}X' \xrightarrow{h q}  X] - [E \xrightarrow{hiq'}  X] -  [X'\xrightarrow{h}  X] + [S\xrightarrow{hi}  X] \tag{rbl}
 \end{equation}
for any such diagram, and define
$$\bK_0(\m V/X \xrightarrow{f}  Y) := \frac{\bM(\m V/X \xrightarrow{f}  Y)}
{\mathbb {BL}(\m V/X \xrightarrow{f}  Y)}.$$
The corresponding equivalence class of $[V \xrightarrow{p}  X]$ shall be denoted by $\Bigl[[V \xrightarrow{p}  X] \Bigr].$
\end{defn}

Note that by Lemma \ref{key-lemma} (1) $f \circ h \circ q$ and $f \circ h \circ i\circ q'$ are smooth (with $Bl_{S}X'$ and $E$ quasi-projective in the case $\m V=\m V^{qp}_k$), so that 
the ``relative blow-up relation'' (\ref{eq:rbl}) makes sense in 
$\bM(\m V/X \xrightarrow{f}  Y)$.

\begin{thm}\label{theorem}
Let $\m V=\m V^{(qp)}_k$ be the category of (quasi-projective) algebraic varieties 
(i.e. reduced separated schemes of finite type) over a base field $k$ of any characteristic, or let $\m V=\m V^{an}_c$ be the category of compact reduced complex analytic spaces.

$\bK_0(\m V/X \xrightarrow{f}  Y)$ becomes a bivariant theory with the following three operations, so that
the canonical projection 
$\bB q: \bM(\m V/-) \to \bK_0(\m V / - )$ is a Grothendieck transformation.

{\bf Product operation}: For morphisms $f: X \to Y$ and $g: Y
\to Z$, the product operation
$$\bigstar: \bK_0 (\m V/X  \xrightarrow{f}  Y) \otimes \bK_0 (\m V/ Y  \xrightarrow{g}  Z) \to
\bK_0 ( \m V/X  \xrightarrow{gf}  Z)$$
is  defined by
$$\Bigl[[V \xrightarrow{h}  X]\Bigr] \bigstar \Bigl[[W  \xrightarrow{k}  Y]\Bigr]:= \Bigl[[V \xrightarrow{h}  X] \bullet [W  \xrightarrow{k}  Y]\Bigr]$$
and bilinearly extended.

{\bf Pushforward operation}: For morphisms $f: X \to Y$
and $g: Y \to Z$ with $f \in \m Prop$, the pushforward operation
$$f_*: \bK_0  (\m V/ X  \xrightarrow{gf} Z) \to \bK_0  (\m V/ Y  \xrightarrow{g}  Z) $$
is  defined by
$$f_*\left (\, \, \Bigl[[V \xrightarrow{p}  X]\Bigr]\, \, \right ) := \Bigl[ f_*([V  \xrightarrow{p}  X])\Bigr]$$
and linearly extended. 

{\bf Pullback operation}: For an independent square
$$\CD
X' @> g' >> X \\
@V f' VV @VV f V\\
Y' @>> g > Y, \endCD
$$
the pullback operation
$$g^* : \bK_0  (\m V/ X  \xrightarrow{f} Y) \to \bK_0 (\m V/ X'  \xrightarrow{f'} Y') $$
is  defined by
$$
g^*\left (\, \,\Bigl[[V  \xrightarrow{p}  X]\Bigr] \, \, \right )  := \Bigl[g^*([V  \xrightarrow{p}  X])\Bigr]
$$
and linearly extended.
\end{thm}

\begin{proof} It suffices to show the well-definedness of these three operations. \\

(i) $ \Bigl[V \xrightarrow {h}  X]\Bigr]\bigstar  \Bigl[W  \xrightarrow {k}  Y]\Bigr] := \Bigl[V \xrightarrow{h}  X] \bullet [W  \xrightarrow{k}  Y]\Bigr]$ is well-defined: Let 
$$ \alpha = [Bl_{S_1}X'\to X] - [E_1\to X] -  [X'\to X] + [S_1\to X] \in \mathbb {BL}(\m V/ X \xrightarrow{f}  Y)$$
and
$$ \beta = [Bl_{S_2}Y'\to Y] - [E_2\to Y] -  [Y'\to Y] + [S_2\to Y] \in \mathbb {BL}(\m V/ Y \xrightarrow {g}  Z)$$
be given. Then we have
\begin{align*}
& \left ([V \xrightarrow {h}  X] + \alpha \right ) \bullet \left ([W  \xrightarrow {k}  Y] + \beta \right) \\
& =  [V \xrightarrow {h}  X] \bullet [W  \xrightarrow {k}  Y]  + [V \xrightarrow {h}  X] \bullet \beta  + \alp \bullet \left ([W  \xrightarrow {k}  Y] + \beta \right),
\end{align*}
and we show that 
$$[V \xrightarrow{h}  X] \bullet \beta  + \alp \bullet \left ([W  \xrightarrow{k}  Y] + \beta \right)
\in \mathbb {BL}(\m V/ X \xrightarrow{g \circ f}  Z).$$
For this end it suffices to show that 
$$[V \xrightarrow{h}  X] \bullet \beta \in \mathbb {BL}(\m V/ X \xrightarrow{g \circ f}  Z)$$
and
$$ \alp \bullet [H  \xrightarrow{j}  Y]  \in \mathbb {BL}(\m V/ X \xrightarrow{g \circ f}  Z)$$
for any $[H  \xrightarrow{j}  Y] \in \bM(\m V/ Y \xrightarrow{g}  Z)$.

For the proof of $ \alp \bullet [H  \xrightarrow{j}  Y]  \in \mathbb {BL}(\m V/ X \xrightarrow{g \circ f}  Z)$, consider the following diagram:
$$
\xymatrix{
& \widetilde {E_1}  \ar[ld] \ar'[d][dd]_{\widetilde{q'}} \ar [rr]^{\widetilde {i'}} && Bl_{\widetilde{S_1}}\widetilde{X'}  \ar[ld] \ar[dd]^{\widetilde q} \\
E_1 \ar[dd]_{q'} \ar[rr]^(.65){i'} && Bl_{S_1}X'  \ar[dd]^(.65){q} \\
& \widetilde {S_1} \ar[ld] \ar'[r][rr]^{\widetilde i}  && \widetilde {X'} \ar[ld] \ar[r]^{\widetilde {h}} & \,\,\widetilde {X} \ar[ld]^{k} \ar[r] & \, \, H \ar[ld]^{j}\\
S_1 \ar[rr]_{i} && X' \ar[r]_{h} &  X \ar[r]_{f} & Y \ar[r]_{g} & Z,}
$$
which by Lemma \ref{key-lemma} (2) is the pullback by the proper morphism $j: H \to Y$
of the following blow-up diagram:
\begin{equation}\label{bd}
\begin{CD}
E_1 @> i'>> Bl_{S_1}X' \\
@V q' VV @VV q V \\ 
S_1 @> i >> X' @> h >> X  @> f >> Y.
\end{CD} 
\end{equation}
Then we have that
\begin{align*}
& \alp \bullet [H  \xrightarrow{j}  Y] \\
& = [Bl_{\widetilde{S_1}}\widetilde{X'} \xrightarrow{k\widetilde{h}\widetilde{q}}  X] - [\widetilde {E_1} \xrightarrow{k\widetilde{h} \widetilde{q} \widetilde {i'}}  X] - [\widetilde {X'} \xrightarrow {k\widetilde{h}}  X] + [\widetilde {S_1} \xrightarrow{k\widetilde{h}\widetilde {i}}  X], 
\end{align*}
which is in $\bM(\m V/ X \xrightarrow{g \circ f}  Z).$ 
In the same way one gets
 $$[V \xrightarrow{h}  X] \bullet \beta \in \mathbb {BL}(\m V/ X \xrightarrow{g \circ f}  Z).$$
Here we are using the fact that the pullback of the corresponding blow-up diagram for
 $\beta$ under the morphism $fh$ is again a similar blow-up diagram, since $fh$ is smooth and therefore flat.

(ii) The well-definedness of $f_*\left[[V \xrightarrow{p}  X]\right ] :=  \left[[V  \xrightarrow{f \circ p}  Y] \right]$ is obvious.

(iii) $g^*\left[[V  \xrightarrow{p}  X]\right ]  := \left [g^*[V  \xrightarrow{p}  X] \right ]$ is well-defined. The proof based on Lemma \ref{key-lemma} (2) is similar to that of (i) above, so omitted.
\end{proof}

Note that the proof of the well-definedness of the product- and pullback operations above 
used Lemma \ref{key-lemma} (2), as well as the fact that the smooth and therefore flat
pullback of a blow-up diagram is again a blow-up diagram.

\begin{rem} Here we note (cf. \cite{EH}) that in general the pullback of a blow-up is \emph{not} the blow-up of the pullback, i.e., consider the following pullback diagram, which is obtained by pulling back a blow-up diagram by the morphism $\widetilde X \to X$:
$$
\xymatrix{
& \widetilde {E}  \ar[ld] \ar'[d][dd]_{\widetilde{q'}} \ar [rr]^{\widetilde {i'}} &&\widetilde {Bl_{S}X}  \ar[ld] \ar[dd]^{\widetilde q} \\
E \ar[dd]_{q'} \ar[rr]^(.65){i'} && Bl_{S}X  \ar[dd]^(.65){q} \\
& \widetilde {S} \ar[ld] \ar'[r][rr]^{\widetilde i}  && \widetilde {X} \ar[ld] \\
S \ar[rr]_{i} && X}
$$
Then the diagram 
$$\begin{CD}
\widetilde {E} @> \widetilde {i'}>> \widetilde {Bl_{S}X} \\
@VV \widetilde {q'} V  @VV \widetilde {q} V \\
\widetilde {S}  @> \widetilde {i} >> \widetilde {X} 
\end{CD} 
$$
is in general \emph{not} a blow-up diagram, i.e., $\widetilde {Bl_{S}X}$ is not the blow-up of $\widetilde {X}$ along $\widetilde {S}$. A typical example is the situation that $S$ is a point of the $2$-dimensional projective space $X = \bP^2$, $\widetilde X$ is a smooth curve going through the point $S$ and $h:\widetilde X \to X$ is the inclusion map.
\end{rem}

Let us finish this section with the following

\begin{rem} In the case when $Y$ is a point, the blow-up diagram defining 
${\mathbb {BL}(\m V/X \xrightarrow{f}  pt)}$ is nothing but the following:
$$\begin{CD}
E @> i'>> Bl_{S}X' \\
@VV q' V  @VV q V \\
S @> i >> X' @> h >> X \:,
\end{CD} 
$$
such that $h: X' \to X$ is proper, $X'$ and $S$ are nonsingular,
and $q: Bl_{S}X' \to X'$ is the blow-up of $X'$ along $S$ with $q':E \to S$  the exceptional divisor map.

Hence ${\mathbb {BL}(\m V/X \xrightarrow{f}  pt)}$ is nothing but $BL(\m V/X)$, i.e., we have 
by Bittner's theorem
$$\bK_0(\m V /X \to pt) \simeq K_0(\m V/X)$$
 in the compact complex analytic context, as well as in the algebraic context over a base field of characteristic zero. Finally note that we  always have a group homomorphism
 $$\bK_0(\m V /X \to pt) \to K_0(\m V/X)\:,$$
 since $Bl_{S}X'\backslash E\simeq X'\backslash S$ in the diagram above so that
 $$[Bl_{S}X'\to X] - [E\to X]= [X'\to X] - [S\to X] \in K_0(\m V/X)\:.$$
\end{rem}

\section {Motivic bivariant Chern and Hirzebruch class transformations}

Now we are ready to prove the following main theorem, which is about the \emph {motivic bivariant Chern and Hirzebruch class transformations}.

\begin{thm}\label{thm:main2} Let $\m V=\m V^{qp}_k$ be the category of quasi-projective algebraic varieties 
 over a base field $k$ of any characteristic.
\begin{enumerate}
\item[(i)]  There exists a unique Grothendieck transformation
$$mC_y=\La_y^{mot}: \bK_0(\m V_k^{qp}/ - ) \to \bK_{alg}( - )\otimes \bZ[y]$$
satisfying the normalization condition that for a smooth morphism $f: X \to Y$ 
 the following equality holds in $\bK_{alg}(X \xrightarrow {f} Y) \otimes \bZ[y]$:
$$\La_y^{mot}\Bigl(\Bigl[[X  \xrightarrow{\op {id}_X}  X]\Bigr]\Bigr) = \La_y(T^*_f) \bullet  \m O_{f}.$$
\item[(ii)]  Let $T_y : \bK_0(\m V_k^{qp} / - ) \to \bH(-) \otimes \bQ[y]$
be defined as the composition 
$\tau \circ \La_y^{mot}$, renormalized by $\cdot (1+y)^i$ on $\bH^i(-)\otimes \bQ[y]$.
Here $\bH$ is either the operational bivariant Chow group,
or the even degree bivariant homology theory for $k=\bC$, with $\tau$ the corresponding
Riemann-Roch transformation. \\
Then $T_y$ 
is the unique Grothendieck transformation
satisfying the normalization condition that for a smooth morphism $f: X \to Y$  the following equality holds in $\bH(X \xrightarrow {f} Y) \otimes \bQ[y]$:
$$T_y\Bigl(\Bigl[[X  \xrightarrow{\op {id}_X}  X]\Bigr]\Bigr) = T^*_y(T_f) \bullet [f].$$
\end{enumerate}
\end{thm}

\begin{proof} 
Uniqueness follows from
$$\Bigl[[V \xrightarrow{h}  X]\Bigr]=
h_*\Bigl(\Bigl[[V \xrightarrow {\op {id}_V} V]\Bigr]\Bigr)\in 
\bK(\m V/X  \xrightarrow{f}  Y)$$
for $h: V\to X$ a proper morphism with $f \circ h$ smooth.
So we simply define in this case
$$\ga_{c\ell}\Bigl(\Bigr[[V \xrightarrow{h}  X]\Bigr]\Bigr) := 
h_* (c\ell (T_{fh}) \bullet \theta(fh)).$$
Here the Grothendieck transformation $\ga_{c\ell}$ is the following:

\begin{itemize}
\item The  motivic bivariant Chern class transformation in (i) 
$$mC_y=\La_y^{mot}: \bK_0(\m V_k^{qp}/ - ) \to \bK_{alg}( - )\otimes \bZ[y]$$
corresponds to the multiplicative characteristic class
$$c\ell(W):= \La_y(W^*)\in K^0(-)[y] \subset  K^0(-)[y,(1+y)^{-1}] $$ 
given by the total $\lambda$-class of the dual vector bundle,
with $\theta(fh)=\m O_{fh}=[\m O_{V}]$.
\item The bivariant Hirzebruch class transformation in (ii)
 $$T_y : \bK_0(\m V_k^{qp} / - ) \to \bH(-) \otimes \bQ[y]$$
 corresponds to the multiplicative characteristic class
$$c\ell(W):= T^*_y(W)\in \bH^*(-)\otimes \bQ[y] $$ 
given by the Hirzebruch class,
with $\theta(fh)=[fh]$ the relative fundamental class.
\end{itemize}

Moreover, these characteristic classes commute with the corresponding
orientations $\theta$ of a smooth morphism (as already explained before).
So  we only have to show that 
\begin{itemize}
\item the corresponding Grothendieck transformation 
$$\ga_{c\ell}=:\La_y^{mot}: \bM(\m V/X  \xrightarrow{f}  Y) \to \bK(X  \xrightarrow{f}  Y)$$
from Corollary \ref{twisting}
vanishes on the subgroup $\mathbb {BL}(\m V/  X \xrightarrow{f}  Y)$, and 
\item the relation $\ga_{T^*_y}=\tau\circ \La_y^{mot}$ up to the 
renormalization by the multiplication  with $(1+y)^i$ on $\bH^i(-)\otimes \bQ[y]$.
\end{itemize}

(i) \underline {$\La_y^{mot}: \bM(\m V/X  \xrightarrow{f}  Y) \to \bK(X  \xrightarrow{f}  Y)$
vanishes on $\mathbb {BL}(\m V/  X \xrightarrow{f}  Y)$}:
Let us identify the vector bundle $T^*_{fh}$ for the smooth morphism $fh: V\to Y$
with the corresponding locally free sheaf $\Omega^1_{fh}$ of sections given by the relative one-forms, so that
$$\Lambda_y^{mot}([V \xrightarrow{h}  X]) := \sum_{p \geq 0} h_*([\Omega^p_{fh}]
\bullet \m O_{fh}) \cdot y^p\:.$$
Note that by the definition of relative perfectness, $D^b_{id-perf}(V)=D^b_{fh-perf}(V)$ for
the smooth morphism $fh$, so that
$$\bullet \m O_{fh}: \bK(V \xrightarrow{id_V}  V)=K_0(D^b_{id-perf}(V))
\xrightarrow{\sim}
K_0(D^b_{fh-perf}(V))=\bK(V \xrightarrow{fh}  X)\:,$$
with $h_*(\;-\; \bullet \m O_{fh})$ induced by the total direct image
$$Rh_*: D^b_{id-perf}(V)= D^b_{fh-perf}(V)\to D^b_{f-perf}(X)\:.$$

Consider now a blow-up diagram
$$\begin{CD}
E @> i'>> Bl_{S}X' \\
@V q' VV  @VV q V \\
S @>> i > X' @>>h > X  @>> f > Y\:,
\end{CD}$$
with $h$ proper and $i$ a closed embedding such that $fh$ and $fhi$ are smooth.
Then we have by \cite[Chapter IV, Theorem 1.2.1 and (1.2.6) on p.74]{Gros} that the following natural morphisms are quasi-isomorphisms for all $p\geq 0$
(and note that Gros is working in \cite[Chapter IV, \S 1.2]{Gros} with the corresponding 
\emph{relative De Rham complexes}):
\begin{enumerate}
\item[(a)] $\Omega^p_{fh} \xrightarrow{\sim}  R^0q_*\Omega^p_{fhq}$.
\item[(b)] $R^kq_*\Omega^p_{fhq} \xrightarrow{\sim}  i_*R^kq'_*\Omega^p_{fhiq'}$ for all $k\geq 1$.
\item[(c)] $\Omega^p_{fhi} \xrightarrow{\sim}  R^0q'_*\Omega^p_{fhiq'}$.
\end{enumerate}
Here (c) can be checked (\'etale) locally, so that it follows from
\cite[(1.2.6) on p.74 and (4.2.12) on p.23]{Gros}. Moreover all coherent sheaves
$\Omega^p_{fh}, \Omega^p_{fhi}$ and $\Omega^p_{fhiq'}$ for $p\geq 0$ are locally free,
since the corresponding morphisms are smooth.
Similarly all direct image sheaves $R^kq'_*\Omega^p_{fhiq'}$ for $k,p\geq 0$ are locally free,
since $q': E\to S$ is a projective bundle (e.g. compare \cite[(1.2.6) on p.74 and (4.2.12) on p.23]{Gros}). Finally the morphisms $i$ and $q$ are of finite Tor-dimension by 
Lemma \ref{key-lemma} (3), with $i$ exact, so that (a) and (b) resp.(c) can be 
considered as quasi-isomorphisms
in $ D^b_{fh-perf}(X')$ resp. $ D^b_{fhi-perf}(S)$. So one gets for all $p\geq 0$ the following equalities in $\bK_{alg}(X' \xrightarrow{fh}  Y)$:
\begin{align*}
q_* [\Omega^p_{fhq}] - i_*q'_* [\Omega^p_{fhiq'}] &= \sum_{k\geq 0}\: (-1)^k\Bigl(
[R^kq_* \Omega^p_{fhq}] - [i_*R^kq'_* \Omega^p_{fhiq'}]\Bigr)\\
&=[R^0q_* \Omega^p_{fhq}] -[i_*R^0q'_* \Omega^p_{fhiq'}] \\
&=[\Omega^p_{fh}] -i_* [\Omega^p_{fhi}] \:.
\end{align*}
And this implies the needed vanishing result:
\begin{align*}
& \Lambda_y^{mot} \left ([Bl_{S}X' \xrightarrow{hq}   X] - [E \xrightarrow{hiq'} X] - 
[X' \xrightarrow{h}  X] +  [S \xrightarrow{hi}  X] \right ) \\
& = \sum_{p \geq 0} \Bigl( \;h_*q_*([\Omega^p_{fhq}])y^p - h_*i_*q'_*([\Omega^p_{fhiq'}])y^p  
 - h_*([\Omega^p_{fh}])y^p  +  h_*i_*([\Omega^p_{fhi}])y^p \;\Bigr) \\
&= \sum_{p \geq 0} h_* \left ( q_*[\Omega^p_{fhq}] - i_*q'_* [\Omega^p_{fhiq'}] - [\Omega^p_{fh}] + i_*[\Omega^p_{fhi}] \right) y^p  = 0 \:.
\end{align*}

(ii) \underline {Proof of the relation $\ga_{T^*_y}=\tau\circ \La_y^{mot}$ up to  
renormalization}:
By composition with the bivariant Riemann--Roch transformation
$\tau: \bK_{alg}(X \xrightarrow{f}  Y) \to \bH(X \xrightarrow{f}  Y)$,
and extending linearly with respect to the coefficients $\bZ[y]$,
we get a Grothendieck transformation
$$\tau\circ \Lambda_y^{mot}: \bK_0(\m V_k^{qp} / - ) \to \bH(-) \otimes \bQ[y]\:.$$
Similarly, the renormalization $\Psi_{(1+y)}: \bH(-) \otimes \bQ[y]\to \bH(-) \otimes 
\bQ[y,(1+y)^{-1}]$
given by
$$\cdot (1+y)^i: \bH^i(-) \otimes \bQ[y]\to  \bH^i(-) \otimes \bQ[y,(1+y)^{-1}]$$
is a Grothendieck transformation, since $\bH(-)$ is a graded bivariant theory.

Now we show that our looking-for transformation $T_y=\ga_{T_y^*}$ can be defined as
$$T_y := \Psi_{(1+y)}\circ \tau \circ \Lambda_y^{mot} :\bK_0(\m V / -) \to \bH(-) \otimes \bQ[y] \subset \bH(-) \otimes  \bQ[y,(1+y)^{-1}]\:.$$
It suffices to check that for a smooth morphism $f:X \to Y$ 
$$T_y ([X \xrightarrow{id}  X]) = T^*_y(T_f) \bullet [f]
\in  \bH(X \xrightarrow{f}  Y)  \otimes \bQ[y] \:.$$
And this can be seen as follows:
\begin {align*}
 \tau \circ \Lambda_y^{mot} ([X \xrightarrow {\op {id}} X]) 
&= \tau(\lambda_y(T^*_f)\bullet \m O_{f} )\\
& =ch(\lambda_y(T^*_f)) \bullet  \tau(\Cal O_f)\\
& = ch(\lambda_y(T^*_f)) \bullet td(T_f) \bullet [f]
\end{align*}
by  the  \emph{Riemann--Roch formula} 
$$\tau(\Cal O_f)= td(T_f) \bullet [f]\:.$$
Compare with \cite [(*) on p.124]{Fulton-MacPherson}
for  $\bH$ the bivariant homology in case $k=\bC$. For $\bH=CH$ the bivariant Chow group and
$k$ of any characteristic, this follows from \cite[Theorem 18.2]{Fulton-book}, as we explain later on in Remark \ref{strong-todd}. So we get
$$\tau \circ \Lambda_y^{mot} ([X \xrightarrow {\op {id}} X])=
\left( \prod_{j=1}^{\op {rank} T_f} (1+ye^{-\alpha _j}) \prod_{j=1}^{\op {rank} T_f} \frac {\alpha_j}{1 - e^{-\alpha_j}} \right) \bullet [f] \:,$$
with $\alpha_j$ the Chern roots of $T_f$.
Here it should be noted that $[f] \in \bH^{-\op {rank}  T_f}(X \xrightarrow{f} Y)$ by 
\cite[Part II, \S 1.3]{Fulton-MacPherson}) resp. \cite[(1) on p.326]{Fulton-book}.
Moreover, the substitution $\alpha_j\mapsto \alpha_j(1+y)$ corresponds to the renormalization
$$\Psi_{(1+y)}: \bH^*(-)\otimes \bQ[y]\to \bH^*(-)\otimes \bQ[y,(1+y)^{-1}]\:,$$
since $\alpha_j\in \bH^1(-)$. So we get

\begin {align*}
& T_y ([X \xrightarrow {\op {id}} X]) = \Psi_{(1+y)}\circ \tau \circ \Lambda_y^{mot} 
([X \xrightarrow {\op {id}} X])\\
& = \left (\prod_{j=1}^{\op {rank} T_f} \left(1+ye^{-\alpha _j(1+y)} \right)
 \frac {\alpha_j (1+y)}{1 - e^{-\alpha_j(1+y)}} \right) \bullet [f]\cdot(1+y)^{-\op {rank} T_f}\\
& = \left (\prod_{j=1}^{\op {rank} T_f} \left(1+ye^{-\alpha _j (1+y)}\right)
 \frac {\alpha_j }{1 - e^{-\alpha_j(1+y)}} \right)  \bullet [f]\\
& = \left (\prod_{j=1}^{\op {rank} T_f} \left(\frac{\alpha_j(1 +y)}{1 - e^{-\alpha_j(1 +y)}} - \alpha_jy \right)\right)  \bullet [f]\\
& = T^*_y(T_f) \bullet [f] \in \bH(X \xrightarrow{f}  Y) \otimes \bQ[y]\:.
\end{align*}
\end{proof}

\begin{rem}
(1) Our construction of the Grothendieck transformation $mC_y=\La_y^{mot}: \bK_0(\m V_k^{qp}/ - )$ $\to \bK_{alg}( - )\otimes \bZ[y]$
based on \cite[Chapter IV, Theorem 1.2.1 and (1.2.6)]{Gros},
 i.e. on the properties (a),(b) and (c) in the proof above,
 also works in the algebraic context without considering only quasi-projective varieties,
if one uses the more sophisticated definition of $\bK_{alg}(X \xrightarrow {f} Y)=K_0(D^b_{f-perf}(X))$ as the Grothendieck goup of the triangulated category of $f$-perfect complexes. 

And a similar definition can also be used in the context of compact complex analytic varieties (compare \cite[Part I, \S 10.10]{Fulton-MacPherson} and \cite{Levy2}). Then it seems reasonable that one can also construct in a similar way in this compact complex analytic context the Grothendieck transformation $mC_y=\La_y^{mot}$. Here it would be enough to prove the analogues of the properties (a), (b) and (c) in the complex analytic context.\\
(2) Similarly one would like to further construct in this compact complex analytic context also the Grothendieck transformation
$T_y$ based on Levy's  K-theoretical Riemann-Roch transformation $$\alpha : \bK_{alg}(-)\to \bK^{top}_0(-)$$ from algebraic to topological
bivariant K-theory (see  \cite{Levy2}), together with the topological bivariant 
Riemann-Roch transformation $$\bK^{top}_0(-)\to \bH(-)\otimes \bQ$$ from \cite[Part I, Example 3.2.2]{Fulton-MacPherson}.
 A key result missing so far is the counterpart $\alpha(\m O_f)=\theta(f)$ of \cite[Part II, Theorem 1.4 (3)]{Fulton-MacPherson},
that $\alpha$ identifies for a smooth morphism $f: X\to Y$ the 
orientation $\m O_f:=[\m O_X]\in \bK_{alg}(X \xrightarrow {f} Y)$ with the
orientation $\theta(f)\in \bK^{top}_0(X \xrightarrow {f} Y).$
\end{rem}

Comparing the different normalization conditions for a smooth morphism
$f:X\to Y$, from Theorem \ref{thm:main2} one gets the following corollary:

\begin{cor}\label{cor-Grothendieck2} Let $\m V=\m V^{qp}_k$ be the category of quasi-projective algebraic varieties 
 over a base field $k$ of any characteristic.
Then we have the following commutative diagrams of Grothendieck transformations:
\begin{enumerate}
\item[(i)] $$\xymatrix{
&   \bK_0(\m V_k^{qp} / - )  \ar [dl]_{mC_0} \ar [dr]^{{T_{0}}} \\
{\bK_{alg}( - ) } \ar [rr] _{\tau}& &  \bH(-) \otimes \bQ.}$$
\item[(ii)] $$\xymatrix{
&   \bK_0(\m V_k^{qp} / - )  \ar [dl]_{\epsilon} \ar [dr]^{{T_{-1}}} \\
{\tilde{\bF}( - ) } \ar [rr] _{\gamma}& &  CH(-) \otimes \bQ,}$$
if $k$ is of characteristic zero. Here $\epsilon$ is the unique Grothendieck transformation
satisfying the normalization condition $\epsilon\Bigl(\Bigl[[X  \xrightarrow{\op {id}_X}  X]\Bigr]\Bigr)=\jeden_f$ for a smooth morphism $f: X \to Y$. And similarly for the bivariant Chern class transformation
$\gamma: \bF( - ) \to A^{PI}( - )\otimes \bQ \supset  CH(-) \otimes \bQ$ in case $k=\bC$.
\item[(iii)] Assume $k$ is of characteristic zero. Then the associated covariant transformations in Theorem \ref{thm:main2} (i) and (ii) 
agree under the identification 
$$\bK_0(\m V_k^{qp}/X\to pt) \simeq K_0(\m V_k^{qp}/X)$$
with  the motivic Chern and Hirzebruch class transformations $mC_y$ and ${T_y}_*$.
\end{enumerate}
\end{cor}

\begin{proof} Everything follows from the different normalization conditions for a smooth morphism $f:X\to Y$, except (ii). First we explain the existence of the Grothendieck transformation
$$\epsilon: \bK_0(\m V_k^{qp} / - )\to \tilde{\bF}( - ) $$
to Ernstr\"{o}m--Yokura's bivariant theory of constructible functions,
resp.  in case $k=\bC$ to Fulton--MacPherson's bivariant theory $\bF( - )$ of constructible functions satisfying the local Euler condition.

(a) Let us first consider the last case. Since $f: X \to Y$ is a smooth morphism, it satisfies trivially the local Euler condition so that 
$ \jeden_f:=\jeden_X  \in \bF(X  \xrightarrow{f }  Y)$.
Moreover, $\theta(f):=\jeden_f$ is a 
stable orientation for the smooth morphism $f$, which commutes with the trivial multiplicative characteristic class
$c\ell(V):=1_X \in \bF(X  \xrightarrow{id_X}  X)$ of a vector bundle $V$ on $X$.
So by Theorem \ref{univ}, we get a unique Grothendieck transformation
$$\epsilon: \bM(\m V/-) \to \bF(-)$$
satisfying  for the smooth morphism $f: X \to Y$ the normalization condition
$$\epsilon([X \xrightarrow {\op {id}_X} X]) = \jeden_f \:.$$
Finally, $\epsilon$ vanishes on the subgroup $\mathbb {BL}(\m V/  X \xrightarrow{f}  Y)$:
Consider a blow-up diagram
$$\begin{CD}
E @> i'>> Bl_{S}X' \\
@V q' VV  @VV q V \\
S @>> i > X' @>>h > X  @>> f > Y\:,
\end{CD}$$
with $h$ proper and $i$ a closed embedding such that $fh$ and $fhi$ are smooth.
Then $q: U':=Bl_{S}X'\backslash E \xrightarrow{\sim} X'\backslash S=:U$ so that $$(fhq)_*\jeden_{fhq}-(fhiq')_*\jeden_{fhiq'}=(fhq)_*1_{U'}=(fh)_*1_{U}
=(fh)_*\jeden_{fh}-(fhi)_*\jeden_{fhi}\:.$$
(b) The same argument works for Ernstr\"{o}m--Yokura's bivariant theory $\tilde{\bF}( - )$,
once we know $ \jeden_f:=\jeden_X  \in \tilde{\bF}(X  \xrightarrow{f }  Y)$ for a smooth morphism
$f: X\to Y$. Consider a fiber square 
$$\begin{CD}
X''@> h'>> X' @> g'>> X \\
@V f'' VV  @V f' VV  @VV f V \\
 Y'' @>> h >  Y' @>> g >  Y \:,
\end{CD}$$
with $h$ and therefore also $h'$ flat.
Then the following diagram is commutative by the \emph{Verdier Riemann-Roch theorem}
for the smooth morphism $f'$ (see \cite{Yokura-VRR}, as well as 
\cite[\S 10.4, p.111]{Fulton-MacPherson} and the proof of \cite[Corollary 2.1 (4)]{BSY1}): 
\begin{equation}\label{VRR-c}
\begin{CD}
F(Y') @> (g^*\jeden_f)\bullet = f'^* >> F(X') \\
@V c_* VV  @VV c_* V \\
CH_*(Y') @>> c(T_{f'})\cap f'^* > CH_*(X') \:.
\end{CD}
\end{equation}
So $\alpha:=\jeden_f$ satisfies the condition ($\tilde{\bF}-1$) of \cite{EY1}
with $c_g(\jeden_f)=c(T_{f'})\cap f'^*$.
But it also satisfies the condition ($\tilde{\bF}-2$) of \cite{EY1}, since
$c(T_{f'})\cap$ commutes with  flat pullback (by \cite[Theorem 3.2(d)]{Fulton-book}) so that
$$h'^*\circ c_g(\jeden_f)=h'^*(c(T_{f'})\cap f'^*)=c(T_{f''})\cap (f''^*\circ h^*)
=c_{g\circ h}(\jeden_f)\circ h^*\:.$$
And this implies $ \jeden_f  \in \tilde{\bF}(X  \xrightarrow{f }  Y)$, together with 
commutativity of the diagram in (ii) by
the following ``strong normalization condition'' for the smooth morphism $f: X\to Y$, which by the definition of the right hand side given in \cite[p.325--326]{Fulton-book} is equivalent to
$c_g(\jeden_f)=c(T_{f'})\cap f'^*$ for all base changes $g$:
\begin{equation}\label{nor-gamma}
\gamma( \jeden_f )= c(T_{f})\bullet [f] \in CH(X  \xrightarrow{f }  Y)\:.
\end{equation}
\end{proof}

\begin{rem}\label{strong-todd} In the same way as above one can get the 
\emph{Riemann--Roch formula} 
$$\tau(\Cal O_f)= td(T_f) \bullet [f]$$
for a smooth (or local complete intersection) morphism $f: X\to Y$ and the bivariant
Riemann-Roch transformation $\tau: \bK_{alg}(-)\to CH(-)\otimes \bQ$ from
\cite[Example 18.3.16]{Fulton-book}. By the definition of $\tau$, the associated covariant transformation $\tau_*$ agrees with the Todd class transformation
$$\tau_*=td_*: G_0(X')\to CH^{-*}(X'\to pt)\otimes \bQ\simeq CH_{*}(X')\otimes \bQ \:,$$
with  the last isomorphism given by  \cite[Proposition 17.3.1]{Fulton-book}.
Since $\tau$ commutes with the bivariant product $\bullet$, one gets for a base change $g$ as above by the \emph{Verdier Riemann-Roch theorem}
\cite[Theorem 18.2] {Fulton-book} a commutative diagram
 \begin{equation}\label{VRR-td}
\begin{CD}
G_0(Y') @> (g^*\Cal O_{f'} )\bullet = f'^* >> G_0(X') \\
@V td_* VV  @VV td_* V \\
CH_*(Y')\otimes \bQ @>> td(T_{f'})\cap f'^* > CH_*(X')\otimes \bQ \:.
\end{CD}
\end{equation}
But $(td_*)\otimes \bQ$ is surjective (in fact even an isomorphism) by 
\cite[Corollary 18.3.2]{Fulton-book}, which implies $\tau_g(\Cal O_{f'})=td(T_{f'})\cap f'^*$
for any such base change $g$. And this is equivalent to the 
Riemann-Roch formula $\tau(\Cal O_f)= td(T_f) \bullet [f]$
by the definition of the right hand side given in \cite[p.325--326]{Fulton-book}.
\end{rem}

Let us finish this paper with the following problem:
We do not know if Brasselet's bivariant Chern class transformation $\ga: \bF(-) \to \bH(-)$ 
to Fulton-MacPherson's bivariant homology $\bH(-)$ (see \cite{Brasselet})
satisfies for a smooth morphism $f:X \to Y$ the ``strong normalization condition'' 
$$\ga(\jeden_f) = c(T_f) \bullet [f] \in \bH(X  \xrightarrow{f}  Y) $$
with $[f]$ the corresponding relative fundamental class.

If this is the case, then Corollary \ref{cor-Grothendieck2} (ii) would also be true for Brasselet's bivariant Chern class transformation $\ga: \bF(-) \to \bH(-)$.\\


\noindent 
{\bf Acknowledgements.} We would like to thank Paolo Aluffi, David Eisenbud, Shihoko Ishii, Toru Ohmoto and Takehiko Yasuda for valuable 
discussions and suggestions. We also would like to thank the referees for some useful comments and suggestions.\\


\end{document}